\documentclass[12pt]{amsart}
\usepackage{amsmath, amsthm, amsfonts, amssymb, graphicx, hyperref, bm,verbatim,mathrsfs,siunitx,natbib,tikz-cd,mathabx,mathtools,dsfont}
\usepackage{stmaryrd,enumerate}
\theoremstyle{plain}
\usepackage{geometry}
\newtheorem{thm}{Theorem}[section]
\newtheorem{lem}[thm]{Lemma}

\newtheorem{cor}[thm]{Corollary}

\numberwithin{sublem}{thm} 
\numberwithin{equation}{section}
\renewcommand{\mod}[1]{\  (\text{mod }{#1})}
\renewcommand*{\d}{\, \textup{d}}
\renewcommand{\gcd}{\operatorname{gcd}}

\newcommand{\e}{\varepsilon}
\newcommand{\Mod}[1]{\ (\mathrm{mod}\ #1)}
\renewcommand{\mod}[1]{\ \mathrm{mod}\ #1}
\numberwithin{equation}{section}
\makeatletter
\@namedef{subjclassname@2020}{%
  \textup{2020} Mathematics Subject Classification}
\makeatother
\usepackage[T1]{fontenc}
\numberwithin{equation}{section}
\let\underbrace\LaTeXunderbrace

\begin{document}
\title{Smoothed Shifted Convolutions of Generalised Divisor Functions}
\author{Cheuk Fung (Joshua) Lau}
\address{Mathematical Institute, University of Oxford, Radcliffe Observatory Quarter, Woodstock Rd, Oxford OX2 6GG, UK}
\email{joshua.cf.lau@gmail.com}
\begin{abstract}
We prove an asymptotic formula for the smoothed shifted convolution of the generalised divisor function $d_k(n)$ and the divisor function $d(n)$ for $k \ge 4$, with a power-saving error term whose exponent is independent of $k$. In particular, for sufficiently large $k$, this improves on the result of Topacogullari (2018).
\end{abstract}
\maketitle
\tableofcontents
\section{Introduction}
Prime numbers are of great interest in number theory, and we may use the von Mangoldt function $\Lambda$ to study patterns of prime numbers. For example, to study the number of twin primes less than or equal to $x$, it suffices to estimate $\sum_{n \leq x} \Lambda(n)\Lambda(n+2)$. One idea for estimating this quantity is to use the following decomposition of Linnik \cite{linnik1963dispersion} 
    \[
    \Lambda(n)=\log n \sum_{k=1}^\infty \frac{(-1)^{k-1}}{k} \widetilde{d}_k(n),
    \]
where $\widetilde{d}_k(n)=\#\{n=n_1 \cdots n_k:n_1,\ldots,n_k>1\}$. Therefore, if one could handle all sums of the form $\sum_{n \leq x} d_k(n) d_j(n+h)$, then we could estimate the number of twin primes. Heuristically, since
we expect $\sum_{n \leq x} d_k(n)$ to grow like $x(\log x)^{k-1}$, one might expect $\sum_{n \leq x} d_k(n) d_j(n+h)$ to behave like $x(\log x)^{k+j-2}$ when $h \neq 0$. The case $j=k=2$ is known as the binary additive divisor problem, which has been studied by many authors. It is known that for $h \in \mathbb{Z}$ non-zero and $|h| \ll x^{\frac23}$,
\[
\sum_{n \leq x} d(n)d(n+h)=xP_{2,h}(\log x)+O(x^{\frac23+\e}),
\]
where $P_{2,h}(t)$ is a quadratic polynomial depending on $h$. This can be found in Motohashi \cite{motohashi1994binary}, where a more detailed review of this problem can be found as well. For $k=3$, Topacogullari \cite{topacogullari2016shifted} proved for $h \in \mathbb{Z}$ non-zero and $|h| \ll x^{\frac23}$,
\[
\sum_{n \leq x} d_3(n)d(n+h)=xP_{3,h}(\log x)+O(x^{\frac89+\e}),
\]
where $P_{3,h}(t)$ is a cubic polynomial depending on $h$. In the general case $k \geq 4$, the first result with a main term of this kind was obtained by Motohashi \cite{motohashi1980asymptotic} using the dispersion method, namely
\[
\sum_{n \leq x} d_k(n)d(n+1)=xP_{k}(\log x)+O_{k}(x(\log \log x)^{c_k} (\log x)^{-1}),
\]
where $c_k$ is a constant depending only on $k$ and $P_{k}(t)$ is a polynomial of degree $k$. Using spectral methods, Drappeau  \cite{drappeau2017sums} obtained a power-saving error term and proved that there exists $\delta>0$ such that for $h \in \mathbb{N}$ and $h \ll x^\delta$, we have
\[
\sum_{n \leq x} d_k(n)d(n+h)=xP_{k,h}(\log x)+O_{k}(x^{1-\frac{\delta}{k}}),
\]
where $P_{k,h}(t)$ is a polynomial of degree $k$ depending on $h$. This result was improved by Topacogullari \cite{topacogullari2018shifted}, who proved that for $h$ non-zero and $|h| \ll x^{\frac{15}{19}}$,
\[
\sum_{n \leq x} d_k(n)d(n+h)=xP_{k,h}(\log x)+O_{k,\e}(x^{1-\frac{4}{15k-9}+\e}+x^{\frac{56}{57}+\e}),
\]
where $P_{k,h}(t)$ is a polynomial of degree $k$ depending on $h$. If we consider a smoothed version of this problem, Topacogullari \cite{topacogullari2018shifted} proved that for $w:[1/2,1] \to [0,\infty)$ a smooth and compactly supported function, if $h \in \mathbb{Z}$ is non-zero and $|h| \ll x^{\frac{15}{19}}$, we have
\begin{equation*}
    \sum_{n} w \left( \frac{n}{x} \right) d_k(n)d(n+h)=xP_{k,h,w}(\log x)+O_{k,w,\e}(x^{1-\frac{1}{3k-2}+\e}+x^{\frac{37}{38}+\frac{\theta}{19}+\e}),
\end{equation*}
where $P_{k,h,w}(t)$ is a polynomial of degree $k$ depending on $k$, $h$ and $w$, and $\theta$ is the best bound towards the Ramanujan-Selberg conjecture, i.e. $\theta=7/64$ from Kim and Sarnak \cite{kim2003functoriality}. The above power saving depends on $k$, and the exponent worsens as $k$ increases. In this paper, we obtain an asymptotic where the exponent of the power saving is uniform in $k$.
\begin{thm} \label{thm:intromainthm}
    Let $w:[1/2,1] \to [0,\infty)$ be smooth and compactly supported. Let $k \ge 4$ and let $h$ be a non-zero integer such that $|h| \ll x^{\frac{25}{28}-\eta}$ for some $0<\eta<25/28$. Then,
    \[
    \sum_n w\left( \frac{n}{x} \right) d_k(n) d(n + h)=xP_{k,h,w}(\log x)+O_{k,w,\e}(x^{1-\frac{7}{128} \eta+\e}+x^{\frac{127}{128}+\e}),
    \]
    where $P_{k,h,w}(t)$ is a polynomial of degree $k$ depending on $k$, $h$ and $w$.
\end{thm}
\noindent
For sufficiently large $k$, this improves on Topacogullari \cite{topacogullari2018shifted}. Theorem \ref{thm:intromainthm} is a corollary of the following result.
\begin{thm}
    Let $w:[1/2,1] \to [0,\infty)$ be smooth and compactly supported, $k \ge 4$, $h \in \mathbb{Z}$, $\e>0$, and $x \in \mathbb{R}^+$ sufficiently large in terms of $\e$. Then for $0 \leq \delta \leq \frac{1}{16}$ and $|h| \ll x^{1-\e}$, we have
    \begin{align*}
        \sum_n w \left( \frac{n}{x} \right) &d_k(n) d(n + h)= \ \int_\mathbb{R} w \left( \frac{\xi}{x} \right) P_{k,h}(\log x,\log \xi,\log(\xi+h)) \d \xi\\
        &+O_{k,w,\e,\delta}\left(x^{1-\delta+2 \delta \theta+\e}\left( 1+\frac{|h|^{\frac{1}{4}}}{x^{\frac14-\frac12 \delta}} \right)+x^{1-\delta+\frac{\theta}{2}+\e} \left(1+\frac{|h|^{\frac{\theta}{2}}}{x^{\frac{\theta}{3}+\frac{2 \delta}{3} \theta}}\right) \right),
    \end{align*}
    where $P_{k,h}$ is a polynomial of degree $k$ depending only on $k$ and $h$, and the implied constant depends on $k,w,\e,\delta$.
\end{thm}
\noindent
We remark that a fixed power saving for the sharp cutoff problem is currently out of reach, since it remains unknown whether an asymptotic formula with a fixed power saving uniform in large $k$ holds for $\sum_{n \le x} d_k(n)$.
\section{Outline}
In this section, we outline the main ideas and set $h=1$ for simplicity. The first few steps are similar to the treatment in Topacogullari \cite{topacogullari2018shifted}. We first write
\[
\sum_n w \left( \frac{n}{x} \right) d_k(n)d(n+1)=\sum_{a_1,\ldots,a_k} w \left( \frac{a_1 \cdots a_k}{x} \right) d(a_1 \cdots a_k+1),
\]
and so it suffices to estimate
\begin{equation} \label{eqn:outlinedyadic}
    \sum_{\substack{a_1,\ldots,a_k\\a_i \asymp A_i}} w \left( \frac{a_1 \cdots a_k}{x} \right) d(a_1 \cdots a_k+1),
\end{equation}
where $a_1,\ldots,a_k$ are supported dyadically on $a_i \asymp A_i$. If there is one large variable, or a product of two variables is large, then (\ref{eqn:outlinedyadic}) can be estimated using the Voronoi formula by splitting the summation into $\mod \prod_i a_i$. Otherwise, some product of the variables $a_i$ must be of 'medium' size, say $a_1 \cdots a_r$. In this case, let $$\Phi_r(b) := \sum_{a_i \asymp A_i \forall i \leq r} d(a_1 \cdots a_rb+1),$$ and write $\widetilde{\Phi}_r(b)$ as the expected main term (using the Voronoi formula). Therefore, (\ref{eqn:outlinedyadic}) becomes
\begin{equation*}
    \sum_{a_i \asymp A_i \forall i >r} \widetilde{\Phi}_r(a_{r+1} \cdots a_k)-\sum_{a_i \asymp A_i \forall i>r} (\widetilde{\Phi}(a_{r+1} \cdots a_k)-\Phi(a_{r+1} \cdots a_k)).
\end{equation*}
The first term is the main term and can be computed easily, and to upper bound the second term we use Cauchy-Schwarz and it suffices to bound
\[
\sum_{b \asymp B} (\widetilde{\Phi}_r(b)-\Phi_r(b))^2=\sum_{b \asymp B} \widetilde{\Phi}_r(b)^2-2\sum_{b \asymp B} \widetilde{\Phi}_r(b)\Phi_r(b)+\sum_{b \asymp B} \Phi_r(b)^2,
\]
where $B=A_{r+1} \cdots A_k$. It is straightforward to estimate the first two sums, while if we open the square in the last sum it suffices to estimate sums of the form
\begin{equation} \label{eqn:outlinecertaindivisor}
    \sum_n w_1 \left( \frac{r_1n}{x} \right) w_2 \left( \frac{r_2n}{x} \right)d(r_1n+1)d(r_2n+1),
\end{equation}
where $r_1,r_2$ are suitably sized with $r_1 \neq r_2$. In fact, before applying Cauchy-Schwarz we can group square factors together, so it suffices to estimate (\ref{eqn:outlinecertaindivisor}) for squarefree $r_1,r_2$.
\\

To do this, we use Theorem \ref{thm:grimmelt10.1}, which is Theorem 10.1 of Grimmelt and Merikoski \cite{grimmelt2024twisted}. This result was proven by spectral methods, and it counts the solutions to determinant equations $ad-bc=h$ twisted by periodic weights. The resulting error term consists of data concerning the ranges of the variables $a,c,d$, as well as a quantity $\mathcal{K}$ that depends on the periodic weight and its underlying geometry.\\

To outline our strategy, we focus on the particular case $r_1,r_2$ coprime, and assume the Ramanujan-Petersson conjecture. We can rewrite (\ref{eqn:outlinecertaindivisor}) as
\begin{equation} \label{eqn:outlinetransformedcertaindivisor}
    \sum_n \sum_{\substack{ad=r_1n+1\\bc=r_2n+1}} w_1 \left( \frac{ad-1}{x} \right) w_2 \left( \frac{bc-1}{x} \right).
\end{equation}
Rewriting the constraints as determinant equations, note $ad=r_1n+1$ and $bc=r_2n+1$ together imply $r_2ad-r_1bc=r_2-r_1$. Also, since $(r_1,r_2)=1$, $r_2ad-r_1bc=r_2-r_1$ implies both $ad=r_1n+1$ and $bc=r_2n+1$. Therefore (\ref{eqn:outlinetransformedcertaindivisor}) becomes
\begin{align*}
    \sum_{r_2ad-r_1bc=r_2-r_1} w_1 \left( \frac{ad-1}{x} \right) w_2 \left( \frac{bc-1}{x} \right)
    = &\sum_{ad-bc=r_2-r_1} w_1 \left( \frac{ad-r_2}{r_2x} \right) w_2 \left( \frac{bc-r_1}{r_1x} \right) \mathds{1}_{r_2 \mid a} \mathds{1}_{r_1 \mid c}
\end{align*}
and let $\alpha(\begin{psmallmatrix}
    a & b\\c &d
\end{psmallmatrix})=\mathds{1}_{r_2 \mid a} \mathds{1}_{r_1 \mid c}$. Then, we dyadically split the variables into $a \asymp A$, $b \asymp B$, $c \asymp C$, $d \asymp D$ and without loss of generality assume $r_1^{-1}C \ll D \ll r_2^{-1}A \ll B$. To treat the very skewed ranges $A>r_2C$ or $D>C$, we apply Poisson summation.\\

In other ranges, we may apply Theorem \ref{thm:grimmelt10.1} with $\Gamma=\Gamma_2(r_2,r_1)$. Then, it is straightforward to compute the main term, so we focus on the error term, in particular $\mathcal{K}$. To bound
\[
\sum_{0 \leq |c| \leq \frac{6C}{D}} \left| \sum_{\tau \in \Gamma \backslash \operatorname{SL}_2(\mathbb{Z})} \alpha(\tau) \overline{\alpha(\tau \begin{psmallmatrix}
                1 & 0 \\ c & 1
            \end{psmallmatrix})} \right|,
\]
we use a description of $\Gamma_2(r_2,r_1) \backslash \operatorname{SL}_2(\mathbb{Z})$ with projective lines to get
\[
\sum_{0 \leq |c| \leq \frac{6C}{D}} \left| \sum_{\tau \in \Gamma \backslash \operatorname{SL}_2(\mathbb{Z})} \alpha(\tau) \overline{\alpha(\tau \begin{psmallmatrix}
                 1 & 0 \\ c & 1
            \end{psmallmatrix})} \right| \ll \frac{C}{r_1r_2D}+1 \ll 1,
\]
and similarly for the corresponding sum over $b$. Combining these bounds, we get $\mathcal{K} \ll 1$. Putting these together, one gets an asymptotic for (\ref{eqn:outlinecertaindivisor}) with error term $\ll \sqrt{r_2x}$. Combining the above and assuming the Ramanujan-Petersson conjecture, a sketch of what we get is
\[
\sum_{\substack{a_1,\ldots,a_k\\ a_i \asymp A_i}} w \left( \frac{a_1 \cdots a_k}{x} \right) d(a_1 \cdots a_k+1)=\mathrm{MT}+R,
\]
where MT is the main term which may be handled straightforwardly, and $R$ is the error term with the upper bounds
\begin{align}
    R &\ll \frac{x^{\frac32+\e}}{A_1^{\frac32}}, \label{eqn:outlineremainderbound1}\\
    R &\ll \frac{x^{\frac32+\e}}{A_1A_2} \left( 1+\frac{A_2^{\frac12}}{A_1^{\frac12}} \right) \left( 1+ \frac{A_1^\frac12 A_2^{\frac12}}{x^{\frac12}} \right), \label{eqn:outlineremainderbound2}\\
    R &\ll \frac{x^{1+\e}}{A^{\frac12}}+A^{\frac34}x^{\frac34+\e}, \label{eqn:outlineremainderbound3}
\end{align}
where $A$ is an arbitrary product of factors $A_1,\ldots,A_k$. As mentioned above, we use the three bounds depending on the sizes of the factors $A_1,\ldots,A_k$. Without loss of generality, assume $A_1 \geq \cdots \geq A_k$, and for $\delta>0$ we let the boundary values to be
\[
X_1=x^{\frac13+\frac23\delta}, \quad X_2=x^{\frac12+\delta}, \quad X_3=x^{\frac13-\frac43\delta}, \quad X_4=x^{2 \delta}.
\]
If $A_1 \gg X_1$, we use (\ref{eqn:outlineremainderbound1}) to get $R \ll x^{1-\delta+\e}$. If $A_1A_2 \gg X_2$, then for $\delta\le 1/2$ we use (\ref{eqn:outlineremainderbound2}) to get
\[
R \ll x^{1-\delta+\e} \left( 1+{x^{-\frac14+\frac12\delta}} \right) \ll x^{1-\delta+\e}.
\]
At last, it can be shown that the only remaining case is $X_4 \ll \prod_{i \in I}A_i \ll X_3$ for some non-empty index set $I \subseteq \{1,2,\ldots,k\}$. Letting $A=\prod_{i \in I}A_i$, using (\ref{eqn:outlineremainderbound3}) we get
\[
R \ll \frac{x^{1+\e}}{X_4^{\frac12}}+X_3^{\frac34}x^{\frac34+\e} \ll x^{1-\delta+\e}.
\]
Compared to Topacogullari \cite{topacogullari2018shifted} our approach differs in two aspects: First, the above application of Theorem \ref{thm:grimmelt10.1} replaces his more classical use of sums of Kloosterman sums. Second, we used a more efficient gluing of variables with Cauchy-Schwarz, as described. The second change alone would allow us to obtain a variant of Theorem \ref{thm:intromainthm} with fixed, albeit worse, power saving.
\section{Acknowledgments}
We would like to thank Jori Merikoski and Lasse Grimmelt for suggesting this question, and for numerous helpful comments throughout the writing of this paper.
\section{Notation}
Throughout this paper, we 
say $f \ll g$ and $f=O(g)$ when there exists a constant $C>0$ such that $|f(x)| \leq Cg(x)$ for $x$ sufficiently large. If this depends on parameter $\e$ say, then we write $f \ll_\e g $ or $f=O_\e(g)$. We use $f=o(g)$ to mean $\lim_{x \to \infty} f(x)/g(x)=0$.\\
\centerline{}
\noindent
Given integers $d_1,d_2$ we use $\gcd(d_1,d_2)$ or $(d_1,d_2)$ to denote the greatest common divisor of $d_1$ and $d_2$, and $\operatorname{lcm}(d_1,d_2)$ or $[d_1,d_2]$ to denote the least common multiple of $d_1$ and $d_2$. We define $(d_1,d_2^\infty) := \prod_{p \mid d_2} p^{\nu_p(d_1)}$.\\
\centerline{}
\noindent
We use $\mathrm{M}_2(\mathbb{Z})$ to denote the set of 2 by 2 matrices with entries in $\mathbb{Z}$, and $\mathrm{M}_{2,h}(\mathbb{Z})$ denotes the subset of $\mathrm{M}_2(\mathbb{Z})$ with determinant $h$. Given $q_1,q_2 \in \mathbb{N}$, we define groups
\begin{align*}
    \Gamma_0(q_1) &:= \left\{\begin{pmatrix}
    a & b\\ c & d
\end{pmatrix} \in \operatorname{SL}_2(\mathbb{Z}):q_1 \mid c\right\},\\
\Gamma_2(q_1,q_2) &:= \left\{ \begin{pmatrix}
            a & b\\ c & d
        \end{pmatrix} \in \operatorname{SL}_2(\mathbb{Z}):q_1 \mid b, \ q_2 \mid c \right\}.
\end{align*}
\centerline{}
\noindent
Given a subgroup $\Gamma \leq \operatorname{SL}_2(\mathbb{Z})$, we say $\alpha:\mathrm{M}_2(\mathbb{Z}) \to \mathbb{C}$ is $\Gamma$-automorphic if $\alpha(\gamma g)=\alpha(g)$ for all $\gamma \in \Gamma$, $g \in \mathrm{M}_2(\mathbb{Z})$. We write $\mathcal{A}(q_1,q_2)$ to be the set of $\Gamma_2(q_1,q_2)$-automorphic functions. We define $\theta$ to be the best bound towards the Ramanujan-Selberg conjecture.\\
\centerline{}
\noindent
For $n,J \in \mathbb{N}$, $\delta>0$ and $X_1,\ldots,X_n$ positive reals, we define the space
\begin{align*}
    C_\delta^J(X_1,\ldots,X_n)=\{&f \in C^J(\mathbb{R}^n):\ \operatorname{supp} f \subseteq [X_1,2X_1] \times \cdots \times [X_n,2X_n],\\&
    \lVert \partial_{x_1}^{J_1} \cdots \partial_{x_n}^{J_n}f \rVert_\infty \leq \prod_{i \leq n} (\delta X_i)^{-J_i} \enspace \forall \, 0 \leq J_1+\cdots+J_n \leq J\}.
\end{align*}
%
\centerline{}
\noindent
Given $p$ prime and $k \in \mathbb{Z}_{>0}$, we define the projective line over  
$\mathbb{Z}/p^k \mathbb{Z}$ by
$$\mathbb{P}_{p^k}^1 := \{(x,y) \in (\mathbb{Z}/p^k \mathbb{Z})^2:x\text{ or }y \in (\mathbb{Z}/p^k \mathbb{Z})^\times\}/\sim,$$
where we define the equivalence relation by $(x_1,y_1)\sim (x_2,y_2)$ if there is $\lambda \in (\mathbb{Z}/p^k \mathbb{Z})^\times$ such that $(x_2,y_2)=(\lambda x_1,\lambda y_1)$. For $q \in \mathbb{Z}_{>0}$, we define
\[
\mathbb{P}_q^1:= \prod_{p^k \parallel q} \mathbb{P}_{p^k}^1,
\]
and by the Chinese Remainder Theorem we can identify $\mathbb{P}_q^1$ with
$
\{(x,y) \in (\mathbb{Z}/q\mathbb{Z})^2:\gcd(x,y,q)=1\}/\sim,
$
where $\sim$ is the equivalence relation defined above with $\lambda \in (\mathbb{Z}/q\mathbb{Z})^\times$.

\section{A Certain Divisor Problem}
\noindent
We first prove the following result on a certain divisor sum.
\begin{thm} \label{thm:certaindivisorcoprime}
    Let $r_1,r_2 \in \mathbb{Z}^+$ be squarefree and distinct, and $h \in \mathbb{Z} \setminus \{0\}$ be such that $(h,r_1r_2)=1$. Let $r_0=(r_1,r_2)$, $\widetilde{r}_i=r_i/r_0$ for $i=1,2$, $x,\eta,\e \in \mathbb{R}^+$, and  define $w_1,w_2:\mathbb{R} \to \mathbb{R}$ smooth functions compactly supported on $[1/2,1]$, satisfying $w_1^{(j)}(x),w_2^{(j)}(x) \ll_j x^{j \eta}$ for all $j \geq 0$. Then for $|h| \ll x^{1-\e}$, we have
    \begin{align*}
        \sum_n w_1 \left( \frac{r_1n}{x} \right) &w_2\left( \frac{r_2n}{x} \right) d(r_1n+h) d(r_2n+h)=\textrm{Main Term}\\
        &+O_{\e}\left( x^{\frac12+O(\eta)} r_0^\e r_2^{\frac12}\gcd(\widetilde{r}_2-\widetilde{r}_1,r_0^\infty)^{1+\e} \left( \left(\frac{|h(r_2-r_1)|}{r_0} \right)^\theta+\frac{x^\theta}{r_0^\theta}  \right) \right),
    \end{align*}
    where the main term is given by
    \[
    \int_\mathbb{R} w_1 \left( \frac{r_1 \xi}{x} \right) w_2 \left( \frac{r_2 \xi}{x} \right) P(\log(r_1 \xi+h),\log(r_2 \xi+h)) \d \xi,
    \]
    and $P(X,Y)$ is a quadratic polynomial depending only on $r_1,r_2,h$.
\end{thm}
\noindent The main idea is to use the following special case of Grimmelt and Merikoski \cite[Theorem 10.1]{grimmelt2024twisted}.
\begin{thm} \label{thm:grimmelt10.1}
    Let $q_1,q_2 \in \mathbb{Z}_{>0}$. For non-zero integers $h,k$ denote
    \[
    \mathrm{M}_{2,h,k}(\mathbb{Z}) := \left\{ \begin{pmatrix}
        a & b\\ c & d
    \end{pmatrix} \in \mathrm{M}_{2,hk}(\mathbb{Z}):\gcd(a,c,k)=\gcd(b,d,k)=1 \right\}.
    \]
    Denote
    $$q=q_1q_2, \quad \Gamma=\Gamma_2(q_1,q_2), \quad T=\Gamma \setminus \operatorname{SL}_2(\mathbb{Z}), \quad T_{1,k} := \operatorname{SL}_2(\mathbb{Z}) \backslash \mathrm{M}_{2,1,k}(\mathbb{Z})$$
    and let $\alpha \in \mathcal{A}(q_1,q_2)$. Let $A,C,D,\delta,\eta>0$ with $AD >\delta$ and denote $Z=\max\{A^{\pm 1},C^{ \pm 1},D^{\pm 1},\delta^{-1}\}$. Assume $|hk| \leq (AD)^{1+\eta}$ and $\gcd(h,kq_1q_2)=1$. Let
    $$f \in C_\delta^7 \left( \frac{A}{\sqrt{|hk|}},\frac{C}{\sqrt{|hk|}},\frac{D}{\sqrt{|hk|}} \right).$$
    Assume that for some $\mathcal{K_+}>0$ we have
    $$\frac1k \sum_{\substack{g=\begin{psmallmatrix}
        a&b\\c&d
    \end{psmallmatrix} \in \operatorname{SL}_2(\mathbb{R})\\
    |a|+|b|C/D+|c|D/C+|d| \leq 10}}\left| \sum_{\substack{\sigma_1,\sigma_2 \in T_{1,k}\\ \sigma_2 g \sigma_1^{-1} =: \sigma \in \operatorname{SL}_2(\mathbb{Z})}}\sum_{\substack{\tau \in T}} \alpha(\tau \sigma \sigma_1) \overline{\alpha(\tau \sigma_2)} \right| \ll Z^{O(\eta)} \mathcal{K}_+.$$
    Then
    \begin{align*}
        &\sum_{g=\begin{psmallmatrix}
        a&b\\c&d
    \end{psmallmatrix} \in \operatorname{M}_{2,h,k}(\mathbb{Z})} \alpha(g) f\left( \frac{a}{\sqrt{|hk|}},\frac{c}{\sqrt{|hk|}},\frac{d}{\sqrt{|hk|}} \right) \\
    =&\frac{1}{\zeta(2) q \prod_{p \mid q} (1+p^{-1})} \sigma_1(|h|) \sum_{\tau \in \Gamma \backslash M_{2,1,k}(\mathbb{Z})} \alpha(\tau) \int_{\mathbb{R}^3} f(a,c,d) \frac{\d a \d c \d d}{c}\\
    &+O \left( Z^{O(\eta)} \delta^{-O(1)} (AD)^{1/2} \mathcal{K}_+^{1/2} \left( \mathcal{R}_0+\min_{j \in \{1,2\}} \mathcal{R}_j \right) \right),
    \end{align*}
    where
    \begin{align*}
        \mathcal{R}_0 &= \frac{ A^{1/2}}{q_1^{1/2} C^{1/2}},\\
        \mathcal{R}_1 &= |h|^{\theta} \left( 1+ \left( \frac{CD}{|hk|q_2} \right)^\theta \right) \left( 1+\left(\frac{C}{Aq_2} \right)^{\frac{1}{2}-\theta} \right),\\
        \mathcal{R}_2 &= \left( 1+ \left( \frac{CD}{|k|q_2} \right)^\theta \right) \left( 1+\left(\frac{|h|C}{Aq_2} \right)^{\frac{1}{2}-\theta} \right).
    \end{align*}
\end{thm}
The following description of quotient by subgroups of $\operatorname{SL}_2(\mathbb{Z})$ will be useful later.
\begin{lem} \label{lem:quotientdescription}
    For positive integers $q_1,q_2$ and $q_0=\gcd(q_1,q_2)$, the maps
    \begin{align*}
        \varpi_{q_1,q_2}:\Gamma_2(q_1,q_2) \backslash \operatorname{SL}_2(\mathbb{Z}) &\to \{((a,b),(c,d)) \in \mathbb{P}_{q_1}^1 \times \mathbb{P}_{q_2}^1:(ad-bc,q_0)=1\}\\
        \left[ \begin{pmatrix}
            a & b\\ c & d
        \end{pmatrix} \right] &\mapsto ([(a,b)],[(c,d)])
    \end{align*}
    and
    \begin{align*}
        \varpi_{q_1}:\Gamma_0(q_1) \backslash \operatorname{SL}_2(\mathbb{Z}) \to \mathbb{P}_{q_1}^1, \quad \left[ \begin{pmatrix}
            a & b\\ c & d
        \end{pmatrix} \right] \mapsto [(c,d)]\\
        \varpi_{q_2}':\Gamma_0(q_2)^T \backslash \operatorname{SL}_2(\mathbb{Z}) \to \mathbb{P}_{q_2}^1, \quad \left[ \begin{pmatrix}
            a & b\\ c & d
        \end{pmatrix} \right] \mapsto [(a,b)]
    \end{align*}
    are bijections.
\end{lem}
\begin{proof}
    See the discussions before Lemma 10.2 of Grimmelt and Merikoski \cite{grimmelt2024twisted}.
\end{proof}
We isolate a lemma from the proof of Theorem \ref{thm:certaindivisorcoprime}.
\begin{lem}\label{lem:boundingKeasy}
    Let $r_1,r_2 \in \mathbb{Z}^+$ be distinct positive integers and $h \in \mathbb{Z}$. Define $r_0=\gcd(r_1,r_2)$ and $\widetilde{r}_i=r_i/r_0$ for $i=1,2$. Suppose $r_1,r_2$ are squarefree and $\gcd(h,r_0)=1$. Let $k=\gcd(\widetilde{r}_2-\widetilde{r}_1,r_0^\infty)$, and let $\sigma=\begin{psmallmatrix} A & B \\ C & D \end{psmallmatrix} \in \operatorname{M}_2(\mathbb{Z})$ be such that $\det \sigma=k$ and $\gcd(A,C,k)=\gcd(B,D,k)=1$. Let $\alpha:\operatorname{M}_2(\mathbb{Z}) \to \mathbb{C}$ be given by
    \[
    \alpha\left(\begin{psmallmatrix} a & b \\ c & d \end{psmallmatrix}\right)=\mathds{1}_{\widetilde{r}_2 \mid a} \mathds{1}_{\widetilde{r}_1 \mid c} \mathds{1}_{r_2 \mid ad-h\widetilde{r}_2}.
    \]
    Then, for $\Gamma=\Gamma_2(r_2,r_1)$, the function $x \mapsto \alpha(x\sigma)$ is left $\Gamma$-automorphic, and for any $\e>0$ we have
    \begin{align*}
        \sum_{x \in \Gamma \backslash\operatorname{SL}_2(\mathbb{Z})} \alpha(x\sigma) \ll_\e r_0^{1+\e}.
    \end{align*}
\end{lem}
\begin{proof}
    We first prove the map $x \mapsto \alpha(x\sigma)$ is left $\Gamma$-automorphic. For any $\gamma = \begin{psmallmatrix} a_1 & b_1 \\ c_1 & d_1 \end{psmallmatrix} \in \Gamma_2(r_2, r_1)$, we have $r_2 \mid b_1$ and $r_1 \mid c_1$, with $\det(\gamma)=1$. For a matrix $Y = \begin{psmallmatrix} A' & B' \\ C' & D' \end{psmallmatrix} \in \operatorname{M}_2(\mathbb{Z})$, the entries of $\gamma Y$ are $\begin{psmallmatrix} a_1 A' + b_1 C' & a_1 B' + b_1 D' \\ c_1 A' + d_1 C' & c_1 B' + d_1 D' \end{psmallmatrix}$. Assuming $\alpha(Y)=1$, we have $\widetilde{r}_2 \mid A'$ and $\widetilde{r}_1 \mid C'$. 
    \begin{itemize}
        \item Since $r_2 \mid b_1$ and $\widetilde{r}_2 \mid A'$, we have $\widetilde{r}_2 \mid (a_1 A' + b_1 C')$.
        \item Similarly, since $r_1 \mid c_1$ and $\widetilde{r}_1 \mid C'$, we have $\widetilde{r}_1 \mid (c_1 A' + d_1 C')$.
        \item Expanding the last condition modulo $r_2$, the cross terms $b_1 c_1 C' B'$ and $b_1 d_1 C' D'$ vanish because $b_1 \equiv 0 \pmod{r_2}$. The term $a_1 c_1 A' B'$ also vanishes because $c_1$ is a multiple of $r_1 = r_0 \widetilde{r}_1$, and $A'$ is a multiple of $\widetilde{r}_2$. We are left with $a_1 d_1 A' D'$. Since $a_1 d_1 \equiv 1 \pmod{r_2}$, this is congruent to $A' D' \pmod{r_2}$. Thus, $r_2 \mid (a_1A'+b_1C')(c_1B'+d_1D')-h\widetilde{r}_2$ holds.
    \end{itemize}
    Thus, $\alpha(\gamma x \sigma) = 1$ if $\alpha(x \sigma)=1$. Since $\gamma^{-1} \in \Gamma_2(r_2,r_1)$, we also have $\alpha(x \sigma) = 1$ if $\alpha(\gamma x \sigma)=1$. Since $\alpha(x\sigma)=0$ or $1$, we have $\alpha(\gamma x \sigma) = \alpha(x \sigma)$, confirming that $x \mapsto \alpha(x \sigma)$ is left $\Gamma$-automorphic. Letting $x = \begin{psmallmatrix} a & b \\ c & d \end{psmallmatrix}$, we have $x\sigma = \begin{psmallmatrix} aA+bC & aB+bD \\ cA+dC & cB+dD \end{psmallmatrix}$. Since $r_1,r_2$ are squarefree, $(r_0,\widetilde{r}_1)=(r_0,\widetilde{r}_2)=(\widetilde{r}_1,\widetilde{r}_2)=1$, and by the Chinese Remainder Theorem we factor the sum into local pieces
    \begin{align*}
        \sum_{x \in \Gamma \backslash \operatorname{SL}_2(\mathbb{Z})} \alpha(x\sigma) &= \prod_{p \mid \widetilde{r}_1} \left( \sum_{[(c,d)] \in \mathbb{P}_p^1} \mathds{1}_{p \mid cA+dC} \right) \prod_{p \mid \widetilde{r}_2} \left( \sum_{[(a,b)] \in \mathbb{P}_p^1} \mathds{1}_{p \mid aA+bC} \right)\\
        &\hspace{1.25em} \prod_{p \mid r_0} \left( \sum_{x \in \Gamma_2(p,p) \backslash \operatorname{SL}_2(\mathbb{Z})} \mathds{1}_{p \mid (aA+bC)(cB+dD)-h\widetilde{r}_2} \right),
    \end{align*}
    where we used the fact that for $p \mid \widetilde{r}_2$, the condition $p \mid aA+bC$ alongside $p \mid \widetilde{r}_2$ implies $p \mid (aA+bC)(cB+dD)-h\widetilde{r}_2$.\\

    For $p \mid \widetilde{r}_1$, suppose $p \mid k$. Then $p \mid (\widetilde{r}_2-\widetilde{r}_1)$, which implies $p \mid \widetilde{r}_2$. This contradicts the fact that $\widetilde{r}_1, \widetilde{r}_2$ are coprime, so $p \nmid k$. This implies $\det \sigma = k \not\equiv 0 \pmod{p}$, and so $(A,C) \not\equiv (0,0) \pmod{p}$. Hence the linear constraint $cA+dC \equiv 0 \pmod{p}$ selects exactly 1 point in $\mathbb{P}_p^1$. Analogously, for $p \mid \widetilde{r}_2$, we have $p \nmid k$ and the sum over $\mathbb{P}^1_p$ equals 1.\\

    For the product over $p \mid r_0$, we use Lemma \ref{lem:quotientdescription}. The condition modulo $p$ is equivalently homogenized as
    \[
    (aA+bC)(cB+dD) \equiv h\widetilde{r}_2(ad-bc) \pmod{p}.
    \]
    We bound the number of valid distinct pairs in two cases:\\
    
    \textbf{Case 1:} $p \nmid k$. Then $\det\sigma \not\equiv 0 \pmod p$, so the map $(u,v) \mapsto (u,v)\sigma$ induces a bijection on $\mathbb{P}_p^1$. Substituting distinct points $(a',b') = (aA+bC, aB+bD)$ and $(c',d') = (cA+dC, cB+dD)$, we have $ad-bc \equiv k^{-1}(a'd'-b'c') \pmod{p}$. Setting $H \equiv h\widetilde{r}_2 k^{-1} \pmod{p}$, the condition becomes $a'd'(1-H) \equiv -H b'c' \pmod{p}$.
    \begin{itemize}
        \item If $H \equiv 1 \pmod{p}$, then $b'c' \equiv 0 \pmod{p}$, forcing $b'=0$ or $c'=0$. If $b'=0$, we have $[(a',b')]=[(1,0)]$ and $p$ valid choices for $[(c',d')]$ to remain distinct. Similarly, $c'=0$ yields $p$ choices. These configurations are disjoint, summing to exactly $2p$ pairs.
        \item If $H \not\equiv 1 \pmod{p}$, set $M \equiv -H(1-H)^{-1} \pmod{p}$, yielding $a'd' \equiv M b'c' \pmod{p}$. Setting $a'=0$ yields $c'=0$, equating the points to $[(0,1)]$, which contradicts $a'd'-b'c' \not\equiv 0 \Mod{p}$. Thus $a' \neq 0$, and we set $a'=1$. A distinct non-trivial solution similarly requires $c' \neq 0$, so set $c'=1$. The equation becomes $d' \equiv M b' \pmod{p}$. $a'd'-b'c' \not\equiv 0 \Mod{p}$ implies $Mb' \not\equiv b' \pmod{p}$, giving exactly $p-1$ valid choices for $b'$ (and none if $M \equiv 1 \pmod{p}$). 
    \end{itemize}
    In all Case 1 scenarios, there are at most $2p$ distinct pairs.

    \textbf{Case 2:} $p \mid k$. Here, $\det \sigma \equiv 0 \pmod{p}$. Because $(A,C,k)=(B,D,k)=1$, we have $(A,C) \not\equiv (0,0)$ and $(B,D) \not\equiv (0,0) \pmod{p}$. Thus $\sigma \pmod{p}$ is a matrix of rank exactly 1, factorizable as $\sigma \equiv \begin{psmallmatrix} u \\ v \end{psmallmatrix} (w,z) \pmod{p}$ for non-zero vectors $(u,v)$ and $(w,z)$. The criterion becomes 
    \[
    (au+bv)(cu+dv)wz \equiv h\widetilde{r}_2 (ad-bc) \pmod{p}.
    \]
    Since $\gcd(h,r_0)=1$, we have $h\widetilde{r}_2 \not\equiv 0 \pmod{p}$. If $wz \equiv 0 \pmod{p}$, the left hand side is $0$ but the right hand side is strictly non-zero, yielding $0$ solutions. 
    If $wz \not\equiv 0 \pmod{p}$, define $E \equiv h\widetilde{r}_2 (wz)^{-1} \pmod{p}$. We apply an invertible basis change $P \in \operatorname{GL}_2(\mathbb{F}_p)$ so that $(u,v)P^T = (1,0)$. Rewriting via new points $(\tilde{a},\tilde{b}) = (a,b)P^{-1}$ and $(\tilde{c},\tilde{d}) = (c,d)P^{-1}$, the constraint is $\tilde{a}\tilde{c} \equiv F (\tilde{a}\tilde{d}-\tilde{b}\tilde{c}) \pmod{p}$ with $F = E \det P \not\equiv 0 \pmod{p}$. If $\tilde{a} \equiv 0$, we have $\tilde{c} \equiv 0$, equating both points to $[(0,1)]$, which contradicts $\tilde{a}\tilde{d}-\tilde{b}\tilde{c} \not\equiv 0 \Mod{p}$, so we set $\tilde{a}=1$. Similarly, $\tilde{c} \neq 0$, so we set $\tilde{c}=1$. The relation simplifies to $1 \equiv F(\tilde{d}-\tilde{b}) \pmod{p}$. Given $F \not\equiv 0$, each of the $p$ choices for $\tilde{b}$ fixes $\tilde{d} \equiv \tilde{b} + F^{-1} \pmod{p}$. Because $F^{-1} \not\equiv 0$, we have $\tilde{d} \not\equiv \tilde{b}$. Therefore, there are exactly $p$ solutions.\\

    Uniformly bounding by $2p$, the product evaluates to
    \[
    \sum_{x \in \Gamma \backslash \operatorname{SL}_2(\mathbb{Z})} \alpha(x\sigma) \le \prod_{p \mid r_0} (2p) = 2^{\omega(r_0)} r_0 \ll_\e r_0^{1+\e}.
    \]
\end{proof}
We use Lemma \ref{lem:boundingKeasy} to prove the following.
\begin{lem}\label{lem:boundingK}
Let $r_1,r_2 \in \mathbb{Z}^+$, $h\in \mathbb{Z}$, and $L \in \mathbb{R}^+$. Define $r_0=\gcd(r_1,r_2)$ and $\widetilde{r}_i=r_i/r_0$ for $i=1,2$. Suppose $r_1,r_2$ are squarefree and $\gcd(h,r_0)=1$. Define $k=\gcd(\widetilde{r}_2-\widetilde{r}_1,r_0^\infty)$. Let $\alpha:\operatorname{M}_2(\mathbb{Z}) \to \mathbb{C}$ be given by
    \[
    \alpha(\begin{psmallmatrix}
        a &b \\ c & d
    \end{psmallmatrix})=\mathds{1}_{\widetilde{r}_2 \mid a} \mathds{1}_{\widetilde{r}_1 \mid c} \mathds{1}_{r_2 \mid ad-h\widetilde{r}_2},
    \]
and define
\[
    \mathrm{M}_{2,h,k}(\mathbb{Z}) := \left\{ \begin{pmatrix}
        a & b\\ c & d
    \end{pmatrix} \in \mathrm{M}_{2,hk}(\mathbb{Z}):\gcd(a,c,k)=\gcd(b,d,k)=1 \right\}.
    \]
    Then, for $\Gamma=\Gamma_2(r_2,r_1)$, $T_{1,k} = \operatorname{SL}_2(\mathbb{Z}) \backslash \mathrm{M}_{2,1,k}(\mathbb{Z})$, and $T = \Gamma \backslash \operatorname{SL}_2(\mathbb{Z})$, we have for any $\e>0$,
\[
\sum_{\substack{g=\begin{psmallmatrix}
        a&b\\c&d
    \end{psmallmatrix} \in \operatorname{SL}_2(\mathbb{R})\\
|a|+|b|L+|c|/L+|d| \leq 10}}\left| \sum_{\substack{\sigma_1,\sigma_2 \in T_{1,k}\\ \sigma_2 g \sigma_1^{-1} =: \sigma \in \operatorname{SL}_2(\mathbb{Z})}}\sum_{\substack{\tau \in T}} \alpha(\tau \sigma \sigma_1) \overline{\alpha(\tau \sigma_2)} \right| \ll \left( \frac{k^3 L}{\widetilde{r}_1 \widetilde{r}_2} + k^2 + \frac{k}{L} \right) r_0^{1+\varepsilon}.
\]
\end{lem}
\begin{proof}
We begin by parameterizing the elements of $T_{1,k}$. Since $k = \gcd(\widetilde{r}_2-\widetilde{r}_1, r_0^\infty)$, $k$ divides $\widetilde{r}_2 - \widetilde{r}_1$, ensuring that $\gcd(k, \widetilde{r}_1) = \gcd(k, \widetilde{r}_2) = 1$. The representatives $\sigma_1, \sigma_2 \in T_{1,k}$ can thus be explicitly chosen as
\begin{equation*}
\sigma_i = \begin{pmatrix} 1 & B_i \\ 0 & k \end{pmatrix}, \quad \text{for } 1 \le B_i \le k, \ \gcd(B_i, k) = 1.
\end{equation*}
There are $\phi(k) \le k$ choices for each $\sigma_i$. Since the sum evaluates the condition $\sigma_2 g \sigma_1^{-1} = \sigma \in \operatorname{SL}_2(\mathbb{Z})$, we can exchange the summation over $g$ for a summation over $\sigma$. For fixed $\sigma_1, \sigma_2$, the matrix $g$ is uniquely determined by $g = \sigma_2^{-1} \sigma \sigma_1$. 
Hence, changing the sum over $g$ to a sum over the integer matrices $\sigma$, it suffices to bound
\begin{equation*}
\sum_{\substack{1 \le B_1, B_2 \le k \\ \gcd(B_1, k) = \gcd(B_2, k) = 1}} \sum_{\substack{\sigma = \begin{psmallmatrix} A & B \\ C & D \end{psmallmatrix} \in \operatorname{SL}_2(\mathbb{Z}) \\ g = \sigma_2^{-1} \sigma \sigma_1 \text{ bounded}}} \left| \sum_{\tau \in T} \alpha(\tau \sigma \sigma_1) \overline{\alpha(\tau \sigma_2)} \right|,
\end{equation*}
where $g=\begin{psmallmatrix}
    a & b\\c & d
\end{psmallmatrix}$ bounded means $|a|+|b|L+|c|/L+|d| \le 10$. Next, we parameterize the constraints on $g$ in terms of the integer entries $A, B, C, D$. Multiplying the matrices gives
\begin{equation*}
g = \begin{pmatrix} 1 & -B_2/k \\ 0 & 1/k \end{pmatrix} \begin{pmatrix} A & B \\ C & D \end{pmatrix} \begin{pmatrix} 1 & B_1 \\ 0 & k \end{pmatrix} = \begin{pmatrix} A - \frac{C B_2}{k} & A B_1 + B k - \frac{C B_1 B_2}{k} - D B_2 \\ \frac{C}{k} & \frac{C B_1}{k} + D \end{pmatrix} =: \begin{pmatrix} a & b \\ c & d \end{pmatrix}.
\end{equation*}
Observe that $a, b, c, d \in \frac{1}{k}\mathbb{Z}$, so we instead sum over the integer entries of $\sigma$. The condition $|a| + |b|L + |c|/L + |d| \le 10$ implies the bounds
$$|C| \le 10Lk,\ 
|A - C B_2/k| \le 10,\ 
|C B_1/k + D| \le 10,\ 
|A B_1 + B k - C B_1 B_2/k - D B_2| \le 10/L.$$
Hence, transferring the constraints strictly to the variables $A, B, C, D$, it suffices to bound
\begin{equation*}
\sum_{\substack{1 \le B_1, B_2 \le k}} \sum_{|C| \le 10Lk} \sum_{|A - \frac{CB_2}{k}| \le 10} \sum_{|\frac{CB_1}{k} + D| \le 10} \sum_{\substack{|A B_1 + B k - C B_1 B_2/k - D B_2| \le 10/L\\ AD-BC=1}} \left| \sum_{\tau \in T} \alpha(\tau \sigma \sigma_1) \overline{\alpha(\tau \sigma_2)} \right|.
\end{equation*}

Now we evaluate the indicator functions to obtain divisibility requirements for $C$. Define the integer matrices $W = \tau \sigma_2$ and $W' = \tau \sigma \sigma_1$. Let $\tau = \begin{psmallmatrix} u & v \\ x & y \end{psmallmatrix} \in \operatorname{SL}_2(\mathbb{Z})$. Expanding $W$ yields
\begin{equation}
W = \begin{pmatrix} u & v \\ x & y \end{pmatrix} \begin{pmatrix} 1 & B_2 \\ 0 & k \end{pmatrix} = \begin{pmatrix} u & u B_2 + v k \\ x & x B_2 + y k \end{pmatrix}.
\end{equation}
The condition $\alpha(W) = 1$ implies $\widetilde{r}_2 \mid u$ and $\widetilde{r}_1 \mid x$.
Similarly, substituting $W' = \tau \sigma_2 g = \tau \sigma_2 (\sigma_2^{-1} \sigma \sigma_1) = \tau \sigma \sigma_1$, we have
\begin{equation}
W' = \begin{pmatrix} u & v \\ x & y \end{pmatrix} \begin{pmatrix} A & B \\ C & D \end{pmatrix} \begin{pmatrix} 1 & B_1 \\ 0 & k \end{pmatrix} = \begin{pmatrix} u A + v C & * \\ x A + y C & * \end{pmatrix}.
\end{equation}
The condition $\alpha(W') = 1$ implies $\widetilde{r}_2 \mid (u A + v C)$ and $\widetilde{r}_1 \mid (x A + y C)$.
Since $\widetilde{r}_2 \mid u$, the first relation simplifies to $\widetilde{r}_2 \mid v C$. Because $\det(\tau) = u y - v x = 1$ and $\widetilde{r}_2 \mid u$, it follows that $\gcd(v, \widetilde{r}_2) = 1$, which forces $\widetilde{r}_2 \mid C$. Similarly, $\widetilde{r}_1 \mid x$ simplifies the second relation to $\widetilde{r}_1 \mid y C$, and since $\gcd(y, \widetilde{r}_1) = 1$, we get $\widetilde{r}_1 \mid C$. Because $\gcd(\widetilde{r}_1, \widetilde{r}_2) = 1$, we conclude that $\widetilde{r}_1 \widetilde{r}_2 \mid C$.
We may therefore parameterize $C = m \widetilde{r}_1 \widetilde{r}_2$ for some integer $m$. 
Hence, replacing the sum over $C$ with a sum over $m$, it suffices to bound
\begin{equation*}
\sum_{\substack{1 \le B_1, B_2 \le k}} \sum_{|m| \le \frac{10Lk}{\widetilde{r}_1 \widetilde{r}_2}} \sum_{|A - \frac{m \widetilde{r}_1 \widetilde{r}_2 B_2}{k}| \le 10} \sum_{|\frac{m \widetilde{r}_1 \widetilde{r}_2 B_1}{k} + D| \le 10} \sum_{\substack{|A B_1 + B k - m\widetilde{r}_1 \widetilde{r}_2B_1 B_2/k - D B_2| \le 10/L \\ AD - m \widetilde{r}_1 \widetilde{r}_2 B = 1}} \left| \sum_{\tau \in T} \alpha(W') \overline{\alpha(W)} \right|.
\end{equation*}
Using $\alpha(W') \le 1$, we now bound the innermost sum using Lemma \ref{lem:boundingKeasy}, which gives
\begin{align*}
    \left| \sum_{\tau \in T} \alpha(W') \overline{\alpha(W)} \right| \le \sum_{\tau \in T} \alpha(\tau \sigma_2) \ll_\e r_0^{1+\e}
\end{align*}
for any $\e>0$. Thus, it suffices to upper bound
\begin{equation} \label{eqn:boundingKlastsuffice}
r_0^{1+\varepsilon} \sum_{B_1, B_2} \sum_{|m| \le \frac{10 L k}{\widetilde{r}_1 \widetilde{r}_2}} \sum_{|A - \frac{m \widetilde{r}_1 \widetilde{r}_2 B_2}{k}| \le 10} \sum_{|\frac{m \widetilde{r}_1 \widetilde{r}_2 B_1}{k} + D| \le 10} \sum_{\substack{|A B_1 + B k - m\widetilde{r}_1 \widetilde{r}_2B_1 B_2/k - D B_2| \le 10/L\\ AD-m\widetilde{r}_1 \widetilde{r}_2 B = 1}} 1.
\end{equation}

We evaluate this sum by splitting it into the cases $m = 0$ and $m \neq 0$.
If $m = 0$, then $C = 0$. The determinant condition becomes $AD = 1$, which provides exactly $2$ solutions for the pair $(A, D)$, namely $\pm(1, 1)$. Moreover, for fixed $A, D, B_1, B_2$, the integer parameter $B$ is restricted to an interval of length $20/(Lk)$. Consequently, there are $\ll (1 + \frac{1}{Lk})$ valid choices for $B$. Summing over the $\le k^2$ choices for $(B_1, B_2)$, the $m=0$ configurations contribute:
\begin{equation*}
r_0^{1+\varepsilon} \sum_{B_1, B_2} \sum_{A,D} \sum_{B} 1 \ll r_0^{1+\varepsilon} k^2 \left( 1 + \frac{1}{Lk} \right) \ll \left( k^2 r_0^{1+\varepsilon} + \frac{k}{L} r_0^{1+\varepsilon} \right).
\end{equation*}

If $m \neq 0$, the integer $C$ is non-zero. For any fixed tuple $(m, B_1, B_2)$, the condition $|A - C B_2/k| \le 10$ isolates $A$ to an interval of length $20$, yielding $O(1)$ valid integers $A$. Identically, $|C B_1/k + D| \le 10$ limits $D$ to an interval of length $20$, yielding $O(1)$ integers $D$. With $A, C$, and $D$ fixed, the integer $B$ is uniquely determined by the relation $AD - BC = 1$ (providing at most $1$ valid choice). 
Summing over the bounded range $1 \le |m| \le \frac{10 L k}{\widetilde{r}_1 \widetilde{r}_2}$ and all $B_1, B_2$, the $m \neq 0$ configurations contribute
\begin{equation*}
r_0^{1+\varepsilon} \sum_{B_1, B_2} \sum_{m \neq 0} \sum_{A, D, B} 1 \ll r_0^{1+\varepsilon} k^2 \left( \frac{L k}{\widetilde{r}_1 \widetilde{r}_2} \right) \ll \frac{k^3 L}{\widetilde{r}_1 \widetilde{r}_2} r_0^{1+\varepsilon}.
\end{equation*}
Therefore, \eqref{eqn:boundingKlastsuffice} is bounded by
\begin{equation}
\ll_\epsilon \left( \frac{k^3 L}{\widetilde{r}_1 \widetilde{r}_2} + k^2 + \frac{k}{L} \right) r_0^{1+\varepsilon},
\end{equation}
as required.
\end{proof}

\begin{proof}[Proof of Theorem \ref{thm:certaindivisorcoprime}]
    Splitting the divisor function, we get
    \[
    \sum_n w_1 \left( \frac{r_1n}{x} \right) w_2 \left( \frac{r_2n}{x} \right) d(r_1n+h) d(r_2n+h)=\sum_n \sum_{\substack{ad=r_1n+h\\ bc=r_2n+h}} w_1 \left( \frac{ad-h}{x} \right) w_2 \left( \frac{bc-h}{x} \right).
    \]
    For convenience, let $r_0=(r_1,r_2)$ and $\widetilde{r}_i=r_i/r_0$ for $i=1,2$. Note
    \[
    \begin{cases}
        ad &= r_1n+h\\
        bc &= r_2n+h
    \end{cases} \implies \widetilde{r}_2ad-\widetilde{r}_1bc=h(\widetilde{r}_2-\widetilde{r}_1).
    \]
    We wish to recover $n$ from the right hand side along with another condition. Since $(\widetilde{r}_1,\widetilde{r}_2)=1$, $\widetilde{r}_1 \mid \widetilde{r}_2(ad-h)$ implies $\widetilde{r}_1 \mid (ad-h)$. If we have $r_0 \mid ad-h$, then we can choose
    \[
    n=\frac{ad-h}{r_0 \cdot \widetilde{r}_1}=\frac{ad-h}{r_1},
    \]
    since $(r_0,\widetilde{r}_1)=1$. Therefore,
    \begin{align}
        \sum_n \sum_{\substack{ad=r_1n+h\\ bc=r_2n+h}} w_1 \left( \frac{ad-h}{x} \right) w_2 \left( \frac{bc-h}{x} \right) &=\sum_{\substack{\widetilde{r}_2ad-\widetilde{r}_1bc=h(\widetilde{r}_2-\widetilde{r}_1)\\r_0 \mid ad-h}} w_1 \left( \frac{ad-h}{x} \right) w_2 \left( \frac{bc-h}{x} \right)\notag{}\\
        &= \sum_{\widetilde{r}_2ad-\widetilde{r}_1bc=h(\widetilde{r}_2-\widetilde{r}_1)} w_1 \left( \frac{ad-h}{x} \right) w_2 \left( \frac{bc-h}{x} \right) \mathds{1}_{r_0 \mid ad-h}. \notag{}
    \end{align}
    Let $\psi:\mathbb{R} \to [0,1]$ be a fixed smooth function supported on $[1/2,1]$ which satisfies
    $$\int_\mathbb{R} \psi \left( \frac{1}{x} \right) \frac{\d x}{x}=1.$$
    Inserting this into our sum, we get
    \begin{align*}
        &\sum_{\widetilde{r}_2ad-\widetilde{r}_1bc=h(\widetilde{r}_2-\widetilde{r}_1)} w_1 \left( \frac{ad-h}{x} \right) w_2 \left( \frac{bc-h}{x} \right) \mathds{1}_{r_0 \mid ad-h}\\
        &= \int_{\mathbb{R}^3} \sum_{\widetilde{r}_2ad-\widetilde{r}_1bc=h(\widetilde{r}_2-\widetilde{r}_1)} w_1 \left( \frac{ad-h}{x} \right) w_2 \left( \frac{bc-h}{x} \right) 
        \psi \left( \frac{\widetilde{r}_2a}{A} \right) \psi \left( \frac{\widetilde{r}_1c}{C} \right) \psi \left( \frac{d}{D} \right) \mathds{1}_{r_0 \mid ad-h} \frac{\mathrm{d}A \mathrm{d} C \mathrm{d} D}{ACD},
    \end{align*}
    Denote $B:=(\widetilde{r}_1x+h\widetilde{r}_1)/C$. By swapping variables $ad \leftrightarrow bc$, $a \leftrightarrow d$, $b \leftrightarrow c$, we may assume
    $$\widetilde{r}_1^{-1}C \ll D \ll \widetilde{r}_2^{-1}A \ll B.$$
    Therefore, it suffices to evaluate 
    \begin{align}
        &\int_{\widetilde{r}_1^{-1}C \ll D \ll \widetilde{r}_2^{-1}A}\sum_{\widetilde{r}_2ad-\widetilde{r}_1bc=h(\widetilde{r}_2-\widetilde{r}_1)} w_1 \left( \frac{ad-h}{x} \right) w_2 \left( \frac{bc-h}{x} \right) 
        \psi \left( \frac{\widetilde{r}_2a}{A} \right) \psi \left( \frac{\widetilde{r}_1c}{C} \right) \psi \left( \frac{d}{D} \right) \mathds{1}_{r_0 \mid ad-h} \frac{\mathrm{d}A \mathrm{d} C \mathrm{d} D}{ACD} \notag{}\\
        &= \int_{\widetilde{r}_1^{-1}C \ll D \ll \widetilde{r}_2^{-1}A} \sum_{ad-bc=h(\widetilde{r}_2-\widetilde{r}_1)} w_1 \left( \frac{ad-h\widetilde{r}_2}{\widetilde{r}_2x} \right) w_2 \left( \frac{bc-h\widetilde{r}_1}{\widetilde{r}_1x} \right) \psi \left( \frac{a}{A} \right) \psi \left( \frac{c}{C} \right) \psi \left( \frac{d}{D} \right) \alpha(\begin{psmallmatrix}
            a&b\\ c&d
        \end{psmallmatrix}) \frac{\d A \d C \d D}{ACD},\label{eqn:referforanalogouscertaindivisorcoprime}
        \end{align}
        where $\alpha(\begin{psmallmatrix}
            a&b\\ c&d
        \end{psmallmatrix}) := \mathds{1}_{\widetilde{r}_2 \mid a} \mathds{1}_{\widetilde{r}_1 \mid c} \mathds{1}_{r_2 \mid ad-h\widetilde{r}_2}$. Let $S(A,C,D)$ be the inner sum. If $A>r_2 x^{\eta+\e'} C$ or $D>r_0 x^{\eta+\e'} C$ for some $\e'>0$, we can use Poisson summation to get the correct main term with a negligible error term. We first consider in detail the case $A>r_2 x^{\eta+\e'} C$. Indeed, for a fixed pair $c$ and $d$, the indicator functions and the determinant condition impose the following congruence requirements on $a$:
\[
    ad \equiv h(\widetilde{r}_2-\widetilde{r}_1 )\pmod{c},\quad a \equiv \widetilde{r}_2 \bar{d}h \pmod{\frac{r_2}{(r_0,d)}},
\]
where $\bar{d}$ is the multiplicative inverse of $d$ mod $\frac{r_0}{(r_0,d)}$, while for $d$ we have
\[
    \frac{a}{\widetilde{r}_2} d \equiv h \pmod{r_0}.
\]
Because $(h, r_1r_2)=1$, this last condition along with $(r_0,\widetilde{r}_2)=1$ implies $(ad, r_0) = 1$. Hence, the second congruence simplifies to $a \equiv \widetilde{r}_2 \overline{d}h \pmod{r_2}$, where $\overline{d}$ is the inverse of $d$ modulo $r_0$. Letting $g = (c,d)$, the first congruence requires $g \mid h(\widetilde{r}_2-\widetilde{r}_1)$. Assuming this alongside the necessary consistency condition $(c, r_2) = 1$, the Chinese Remainder Theorem restricts $a$ to a single residue class $\alpha$ modulo $q := r_2 c / g$.\\

Applying Poisson summation to the inner sum over $a$, we have
\begin{align*}
    &\sum_{a \equiv \alpha \Mod q} w_1 \left( \frac{ad-h\widetilde{r}_2}{\widetilde{r}_2 x} \right) w_2 \left( \frac{ad-h\widetilde{r}_2}{\widetilde{r}_1 x} \right) \psi \left( \frac{a}{A} \right)\\
    &= \frac{1}{q} \sum_{\ell \in \mathbb{Z}} \int_{\mathbb{R}} w_1 \left( \frac{ud-h\widetilde{r}_2}{\widetilde{r}_2 x} \right) w_2 \left( \frac{ud-h\widetilde{r}_2}{\widetilde{r}_1 x} \right) \psi \left( \frac{u}{A} \right) e \left( -\frac{\ell(u-\alpha)}{q} \right) \d u.
\end{align*}
For the main term $\ell=0$, we make the substitution $v = ud - h\widetilde{r}_2$ to obtain
\[
    \frac{g}{r_2 c d} \int_{\mathbb{R}} w_1 \left( \frac{v}{\widetilde{r}_2x} \right) w_2 \left( \frac{v}{\widetilde{r}_1x} \right) \psi \left( \frac{v+h\widetilde{r}_2}{dA} \right) \d v.
\]
Summing this over all valid $c$ and $d$ and integrating over $A,C,D$, this yields the main term
\begin{align} \label{eqn:Poissonmainterm}
    \int_{\substack{C \ge 1\\\widetilde{r}_1^{-1}C \ll D \ll \widetilde{r}_2^{-1}A\\A>r_2x^\eta C}}\sum_{\substack{c,d \\ \widetilde{r}_1 \mid c, \, (c, r_2)=1 \\ (d, r_0)=1 \\ (c,d) \mid h(\widetilde{r}_2-\widetilde{r}_1)}} \frac{(c,d)}{r_2 c d} \psi \left( \frac{c}{C} \right) \psi \left( \frac{d}{D} \right) \int_{\mathbb{R}} w_1 \left( \frac{v}{\widetilde{r}_2x} \right) w_2 \left( \frac{v}{\widetilde{r}_1x} \right) \psi \left( \frac{v+h\widetilde{r}_2}{dA} \right) \d v \frac{\d A \d C \d D}{ACD}.
\end{align}
To bound the error term, we consider the integrals for $\ell \neq 0$. Let 
$$F(u) = w_1 \left( \frac{ud-h\widetilde{r}_2}{\widetilde{r}_2x} \right) w_2 \left( \frac{ud-h\widetilde{r}_2}{\widetilde{r}_1x} \right)\psi(u/A).$$ 
The support of $F(u)$ is of length $\ll A$. Since $u \asymp A$ and $ud \asymp \widetilde{r}_2x$, each derivative with respect to $u$ yields a factor of $$\frac{dx^\eta}{\widetilde{r}_2x} \asymp \frac{dx^\eta}{\widetilde{r}_1x} \asymp \frac{x^\eta}{A}.$$
Thus, $F^{(J)}(u) \ll x^{J\eta}A^{-J}$. Integrating by parts $J \ge 2$ times, the integral is bounded by
\[
    \int_{\mathbb{R}} F(u) e \left( - \frac{\ell(u-\alpha)}{q} \right) \d u \ll \left( \frac{q}{|\ell|} \right)^J A \cdot x^{J \eta}A^{-J} = \frac{q^Jx^{J \eta}}{|\ell|^JA^{J-1}}.
\]
Multiplying by the $\frac{1}{q}$ factor and summing over $\ell \neq 0$, the error term for a fixed pair $(c,d)$ is bounded by
\[
    \ll \sum_{\ell \neq 0} \frac{1}{q} \frac{q^Jx^{J \eta}}{|\ell|^J A^{J-1}} \ll \frac{q^{J-1}x^{J \eta}}{A^{J-1}} \le \frac{(r_2 c)^{J-1}x^{J \eta}}{A^{J-1}}.
\]
Summing this contribution over $c \asymp C$ and $d \asymp x/A$, we find that the total error term is bounded by
\[
    \ll \sum_{c \asymp C} \sum_{d \asymp x/A} \psi \left( \frac{c}{C} \right) \psi \left( \frac{d}{D} \right) \frac{(r_2 c)^{J-1}x^{J \eta}}{A^{J-1}} \ll x C \frac{(r_2 C)^{J-1}x^{J \eta}}{A^J}.
\]
By our assumption $A > r_2 x^{\eta+\e'} C$, we have $A^J > (r_2 C)^J x^{J \eta+J\e'}$. Therefore, the total error term is bounded by
\[
    \ll x^{1-J\e'}.
\]
By choosing $J$ sufficiently large, this contribution is $\ll x^{-K}$ for any arbitrarily large $K$, and is thus negligible. For the case $D > r_0 x^{\eta+\e'} C$, the calculations are analogous. By applying Poisson summation on $d$ instead of $a$, we obtain an analogous main term \eqref{eqn:Poissonmainterm}, alongside a similarly negligible error term bounded by $\ll x^{-K}$.\\

We now focus on the complementary range
    \[
    1 \ll \frac{A}{C} \ll r_2 x^{\eta+\e'}, \quad \widetilde{r}_1^{-1} \ll \frac{D}{C} \ll r_0x^{\eta+\e'}.
    \]
    Our goal is to apply Theorem \ref{thm:grimmelt10.1}. However, in the theorem statement the condition $\gcd(h(\widetilde{r}_2-\widetilde{r}_1),r_1r_2)=1$ is assumed, which is not necessarily true since $\widetilde{r}_2-\widetilde{r}_1$ and $r_0$ might share common factors. Therefore, we let $k=\gcd(\widetilde{r}_2-\widetilde{r}_1,r_0^\infty)$. Observe that $ad-bc=h(\widetilde{r}_2-\widetilde{r}_1)$, $r_2 \mid ad-h\widetilde{r}_2$, $(h,r_1r_2)=1$ implies that $(r_0,ad)=1$. Moreover, we also have $r_2 \mid bc-h\widetilde{r}_1$, and so we also have $(r_0,bc)=1$. Therefore, $(a,c,k)=(b,d,k)=1$.  Then, we can rewrite our original sum as
    \begin{align}
        &\int_{\substack{C \ge 1\\C \ll A\ll r_2 x^{\eta+\e'} C\\ \widetilde{r}_1^{-1} C \ll D \ll r_0x^{\eta+\e'} C}} \sum_{\substack{ad-bc=h_0k\\(a,c,k)=(b,d,k)=1}} w_1 \left( \frac{ad-h\widetilde{r}_2}{\widetilde{r}_2x} \right) w_2 \left( \frac{bc-h\widetilde{r}_1}{\widetilde{r}_1x} \right) \notag{}\\
               &\hspace{13em}\mathds{1}_{\widetilde{r}_2 \mid a}\mathds{1}_{\widetilde{r}_1 \mid c} \mathds{1}_{r_2 \mid a d-h\widetilde{r}_2} \psi \left( \frac{a}{A} \right)\psi \left( \frac{c}{C} \right)\psi \left( \frac{d}{D} \right) \label{eqn:certaindivisorcoprimesuffice}
\frac{\mathrm{d}A \mathrm{d} C \mathrm{d} D}{ACD},
    \end{align}
    where $h_0 :=h(\widetilde{r}_1-\widetilde{r}_2)/k$. Now apply Theorem \ref{thm:grimmelt10.1} with the $h$ as $h_0$, and
    \[q_1=r_0\widetilde{r}_2, \quad q_2=r_0\widetilde{r}_1, \quad k =(\widetilde{r}_2-\widetilde{r}_1,r_0^\infty), \quad \Gamma:=\Gamma_2(r_0\widetilde{r}_2,r_0\widetilde{r}_1),
    \]
    and we define $f(x_1,x_2,x_3) \in C_\delta^7 (A,C,D)$  by
    \begin{align*}
        w_1 \left( \frac{|h_0k|x_1x_3-h\widetilde{r}_2}{\widetilde{r}_2x} \right) w_2 \left( \frac{|h_0k|x_1x_3-h\widetilde{r}_2}{\widetilde{r}_1x} \right)\psi \left( \frac{\sqrt{|h_0k|}x_1}{A} \right)\psi \left( \frac{\sqrt{|h_0k|}x_2}{C} \right)\psi \left( \frac{\sqrt{|h_0k|}x_3}{D} \right).
    \end{align*}
    We now check that all conditions of Theorem \ref{thm:grimmelt10.1} are satisfied. 
    \begin{itemize}
        \item By Lemma \ref{lem:boundingKeasy}, $\alpha$ is left $\Gamma$-automorphic.
        \item In \eqref{eqn:certaindivisorcoprimesuffice},
        $ad-h\widetilde{r}_2 \gg \widetilde{r}_2x$, and since $|h| \ll x^{1-\e}$ we have $AD \gg ad \gg \widetilde{r}_2x \asymp \widetilde{r}_1x$. Thus, we have
        \[
        (AD)^{1+\eta} \ge AD \gg |h(\widetilde{r}_2-\widetilde{r}_1)| = |h_0k|.
        \]
        \item By assumption, we have $(h,r_1r_2)=1$, and so $(h,kr_0^2\widetilde{r}_1\widetilde{r}_2)=1$. Also, $(\frac{\widetilde{r}_2-\widetilde{r}_1}{(\widetilde{r}_2-\widetilde{r}_1,r_0^\infty)},r_1r_2)=1$, and so $(\frac{\widetilde{r}_2-\widetilde{r}_1}{(\widetilde{r}_2-\widetilde{r}_1,r_0^\infty)},kr_0^2\widetilde{r}_1\widetilde{r}_2)=1$. This gives $(h_0,kq_1q_2)=1.$
    \end{itemize}
    Therefore, Theorem \ref{thm:grimmelt10.1} may be applied, and we first treat the main term
    \begin{align*}
        &\frac{\sigma_1(|h_0|)}{\zeta(2)r_0^2\widetilde{r}_1\widetilde{r}_2 \prod_{p \mid r_0^2\widetilde{r}_1\widetilde{r}_2} (1+p^{-1})} \sum_{\tau \in \Gamma \backslash M_{2,1,k}(\mathbb{Z})} \alpha(\tau) \int_{C \ge 1} \int_{\substack{C \ll A\ll r_2 x^{\eta+\e'} C\\ \widetilde{r}_1^{-1} C \ll D \ll r_0x^{\eta+\e'} C}} \int_{\mathbb{R}^3} f(a,c,d) \frac{\d a \d c \d d}{c} \frac{\mathrm{d}A \mathrm{d} D \mathrm{d} C}{ADC}\\
        &=C_{h,r_1,r_2,} \int_{C \ge 1} \int_{\substack{C \ll A\ll r_2 x^{\eta+\e'} C\\ \widetilde{r}_1^{-1} C \ll D \ll r_0x^{\eta+\e'} C\\D \ll \widetilde{r}_2^{-1}A}} \int_{\mathbb{R}^3} w_1 \left( \frac{|h_0k|ad-h\widetilde{r}_2}{\widetilde{r}_2x} \right) w_2 \left( \frac{|h_0k|ad-h\widetilde{r}_2}{\widetilde{r}_1x} \right)\\
        &\hspace{16em}\psi \left( \frac{\sqrt{|h_0k|}a}{A} \right)\psi \left( \frac{\sqrt{|h_0k|}c}{C} \right)\psi \left( \frac{\sqrt{|h_0k|}d}{D} \right)\frac{\d a \d c \d d}{c} \frac{\mathrm{d}A \mathrm{d} D \mathrm{d} C}{ADC}\\
        &=C_{h,r_1,r_2}' \int_{C \ge 1} \int_{\substack{C \ll A\ll r_2 x^{\eta+\e'} C\\ \widetilde{r}_1^{-1} C \ll D \ll r_0x^{\eta+\e'} C\\D \ll \widetilde{r}_2^{-1}A}} \int_{\mathbb{R}^3} w_1 \left( \frac{ad-h\widetilde{r}_2}{\widetilde{r}_2x} \right) w_2 \left( \frac{ad-h\widetilde{r}_2}{\widetilde{r}_1x} \right)\\
        &\hspace{15.5em}\psi \left( \frac{a}{A} \right)\psi \left( \frac{c}{C} \right)\psi \left( \frac{d}{D} \right)\frac{\d a \d c \d d}{c} \frac{\mathrm{d}A \mathrm{d} D \mathrm{d} C}{ADC}\\
        &=C_{h,r_1,r_2}' \int_{C \ge 1} \int_{\substack{C \ll A\ll r_2 x^{\eta+\e'} C\\ \widetilde{r}_1^{-1} C \ll D \ll r_0x^{\eta+\e'} C\\D \ll \widetilde{r}_2^{-1}A}} \int_{\mathbb{R}^2} \int_{\mathbb{R}} w_1 \left( \frac{v}{\widetilde{r}_2x} \right) w_2 \left( \frac{v}{\widetilde{r}_1x} \right)\\
        &\hspace{17em}\psi \left( \frac{v+h\widetilde{r}_2}{dA} \right)\psi \left( \frac{c}{C} \right)\psi \left( \frac{d}{D} \right)\d v\frac{\d c \d d}{cd} \frac{\mathrm{d}A \mathrm{d} D \mathrm{d} C}{ADC},
    \end{align*}
    where
    \begin{align*}
        C_{h,r_1,r_2}' = \frac{1}{\zeta(2)r_0\widetilde r_1\widetilde r_2}{\sigma_{-1}\left(\left|{h(\widetilde{r}_2-\widetilde{r}_1)}\right|\right)} \prod_{p \mid r_0} \frac{1-p^{-1}}{1+p^{-1}} \prod_{p \mid \widetilde{r}_1\widetilde{r}_2} \frac{1}{1+p^{-1}}.
    \end{align*}
    We now write \eqref{eqn:Poissonmainterm} in a similar form. Using summation by parts, \eqref{eqn:Poissonmainterm} can be written as
    \begin{align*}
        C_{h,r_1,r_2}' &\int_{C \ge 1} \int_{\substack{A>r_2 x^{\eta+\e'} C\\\widetilde{r}_1^{-1} C \ll D \ll \widetilde{r}_2^{-1}A}} \int_{\mathbb{R}^2} \int_{\mathbb{R}} w_1 \left( \frac{v}{\widetilde{r}_2x} \right) w_2 \left( \frac{v}{\widetilde{r}_1x} \right)\psi \left( \frac{v+h\widetilde{r}_2}{dA} \right)\psi \left( \frac{c}{C} \right)\psi \left( \frac{d}{D} \right)\d v\frac{\d c \d d}{cd} \frac{\mathrm{d}A \mathrm{d} D \mathrm{d} C}{ADC}\\
        &+\int_{\mathbb{R}} w_1 \left( \frac{v}{\widetilde{r}_2x} \right) w_2 \left( \frac{v}{\widetilde{r}_1x} \right) L(\log(\widetilde{r}_2^{-1}v+h),\log(\widetilde{r}_1^{-1}v+h)) \d v+O(r_1^\e x^{1/2}),
    \end{align*}
    where $L(X,Y)$ is a linear polynomial depending only on $r_1,r_2,h$. Together with the analogous main term from $D>r_0x^{\eta+\e'} C$, it suffices to evaluate
    \begin{align} \label{eqn:aftercombinenotyettruncate}
        C_{h,r_1,r_2}' \int_{\mathbb{R}}\int_{C \ge 1} \int_{\substack{\widetilde{r}_1^{-1} C \ll D \ll \widetilde{r}_2^{-1}A}} \int_{\mathbb{R}^2} w_1 \left( \frac{v}{\widetilde{r}_2x} \right) w_2 \left( \frac{v}{\widetilde{r}_1x} \right)\psi \left( \frac{v+h\widetilde{r}_2}{dA} \right)\psi \left( \frac{c}{C} \right)\psi \left( \frac{d}{D} \right)\frac{\d c \d d}{cd} \frac{\mathrm{d}A \mathrm{d} D \mathrm{d} C}{ADC}\d v
    \end{align}
    Let $\kappa_1, \kappa_2 > 0$ be the implied absolute constants corresponding to $\widetilde{r}_1^{-1} C \ll D$ and $D \ll \widetilde{r}_2^{-1}A$ respectively. The integration region for $A, C,$ and $D$ is bounded by $1 \le C \le \kappa_1 \widetilde{r}_1 D$ and $A \ge \kappa_2 \widetilde{r}_2 D$. We define $X_v = \widetilde{r}_2^{-1}v+h$. Note that the integral over $c$ evaluates to 1 by $\int_\mathbb{R} \psi(y) \d y/y=\int_\mathbb{R} \psi(1/x) \d x/x=1$. We let $d = u_2 D$. The integral over $C$ ranges from $1$ to $\kappa_1 \widetilde{r}_1 D$, yielding
\begin{align*}
    \int_1^{\kappa_1 \widetilde{r}_1 D} \frac{\mathrm{d}C}{C} = \log(\kappa_1 \widetilde{r}_1 D).
\end{align*}
\eqref{eqn:aftercombinenotyettruncate} is then reduced to
\begin{align*}
    C_{h,r_1,r_2}' \int_v w_1 \left( \frac{v}{\widetilde{r}_2x} \right) w_2 \left( \frac{v}{\widetilde{r}_1x} \right)\int_\mathbb{R} \int_{A \ge \kappa_2 \widetilde{r}_2D} \int_{\mathbb{R}} \log(\kappa_1 \widetilde{r}_1 D) \psi(u_2) \psi\left( \frac{\widetilde{r}_2 X_v}{u_2 D A} \right) \frac{\mathrm{d}u_2}{u_2} \frac{\mathrm{d}A}{A} \frac{\mathrm{d}D}{D} \d v.
\end{align*}
Next, we evaluate the integral over $A$. For fixed $D$ and $u_2$, we let $t = \frac{\widetilde{r}_2 X_v}{u_2 D A}$. Then, the integral over $A$ is $$\int_0^{\frac{X_v}{\kappa_2 u_2 D^2}} \psi(t) \frac{\mathrm{d}t}{t}.$$ 
We now evaluate the integral over $D$
\begin{align*}
    \int_1^\infty \log(\kappa_1 \widetilde{r}_1 D) \left[ \int_0^{\frac{X_v}{\kappa_2 u_2 D^2}} \psi(t) \frac{\mathrm{d}t}{t} \right] \frac{\mathrm{d}D}{D}.
\end{align*}
We compute this with integration by parts. Let $\mathrm{d}G(D) = \log(\kappa_1 \widetilde{r}_1 D) \frac{\mathrm{d}D}{D}$, which has the antiderivative $G(D) = \frac{1}{2}\log^2(\kappa_1 \widetilde{r}_1 D)$. Let $F(D)$ denote the inner integral over $w$. By the Fundamental Theorem of Calculus and the chain rule, its derivative is
\begin{align} \label{eqn:integraloverD}
    F'(D) = \psi\left( \frac{X_v}{\kappa_2 u_2 D^2} \right) \cdot \frac{\kappa_2 u_2 D^2}{X_v} \cdot \left( \frac{-2 X_v}{\kappa_2 u_2 D^3} \right) = \frac{-2}{D} \psi\left( \frac{X_v}{\kappa_2 u_2 D^2} \right).
\end{align}
Noting $F(\infty)=0$ and $F(1)=1$, \eqref{eqn:integraloverD} becomes
\begin{align*}
    \int_1^\infty \log^2(\kappa_1 \widetilde{r}_1 D) \psi\left( \frac{X_v}{\kappa_2 u_2 D^2} \right) \frac{\mathrm{d}D}{D}.
\end{align*}
We now set $w = \frac{X_v}{\kappa_2 u_2 D^2}$, which gives $D = \sqrt{\frac{X_v}{\kappa_2 u_2 w}}$ and $\frac{\mathrm{d}D}{D} = -\frac{1}{2}\frac{\mathrm{d}w}{w}$, and \eqref{eqn:integraloverD} equals
\begin{align*}
    \frac{1}{2} \int_0^{\frac{X_v}{\kappa_2 u_2}} \log^2\left( \kappa_1 \widetilde{r}_1 \sqrt{\frac{X_v}{\kappa_2 u_2 w}} \right) \psi(w) \frac{\mathrm{d}w}{w}.
\end{align*}
Due to the $\psi(\widetilde{r}_2X_v/(u_2DA))$ factor, $X_v/\kappa_2u_2 > 2$ for sufficiently large $x$, so the above integral equals
\begin{align*}
    \frac{1}{2} \int_0^\infty \log^2\left( \kappa_1 \widetilde{r}_1 \sqrt{\frac{X_v}{\kappa_2 u_2 w}} \right) \psi(w) \frac{\mathrm{d}w}{w}.
\end{align*}
Expanding the logarithm gives
\begin{align*}
    \log\left( \kappa_1 \widetilde{r}_1 \sqrt{\frac{X_v}{\kappa_2 u_2 w}} \right) = \frac{1}{2} \log X_v + \log \widetilde{r}_1 - \frac{1}{2} \log u_2 - \frac{1}{2} \log w + \log\left(\frac{\kappa_1}{\sqrt{\kappa_2}}\right).
\end{align*}
Integrating over $u_2$, for sufficiently large $x$ \eqref{eqn:aftercombinenotyettruncate} equals
\begin{align*}
   \int_{\mathbb{R}} w_1 \left( \frac{v}{\widetilde{r}_2x} \right) w_2 \left( \frac{v}{\widetilde{r}_1x} \right) P\left(\log(\widetilde{r}_2^{-1}v+h),\log(\widetilde{r}_1^{-1}v+h)\right) \mathrm{d} v,
\end{align*}
where $P(X,Y)$ is a quadratic polynomial whose coefficients depend on $h, r_1,$ and $r_2$, as required. We now treat the error term from our application of Theorem \ref{thm:grimmelt10.1}. Using Lemma \ref{lem:boundingK} with $L=C/D \ll \widetilde{r}_1$, we have
    \[
    \mathcal{K}_+ \ll \left( \frac{k^2 \widetilde{r}_1}{\widetilde{r}_1 \widetilde{r}_2} + k + r_0x^\eta \right) r_0^{1+\varepsilon} \ll k^2r_0^{2+\e} x^\eta.
    \]
    We bound $\mathcal{R}_0,\mathcal{R}_1,\mathcal{R}_2$ by
    \begin{align*}
        \mathcal{R}_0 &\ll \frac{A^{1/2}}{r_2^{1/2} C^{1/2}} \ll x^\frac{\eta}{2},\\
        \mathcal{R}_1 &\ll |h(\widetilde{r}_2-\widetilde{r}_1)|^\theta \left( 1+ \left( \frac{CD}{r_1|h(\widetilde{r}_2-\widetilde{r}_1)|} \right)^\theta \right)\left( 1+\left( \frac{1}{r_1} \right)^{\frac12-\theta} \right),\\
        \mathcal{R}_2 &\ll \left( 1+\left( \frac{CD}{r_1} \right)^\theta \right) \left( 1+ \left( \frac{|h(\widetilde{r}_2-\widetilde{r}_1)|}{r_1} \right)^{\frac{1}{2}-\theta} \right).
    \end{align*}
    Since $(AD)^{1/2} \ll \widetilde{r}_2^{1/2} x^{1/2+O(\eta)}$ and $CD \ll AD \ll \widetilde{r}_2x$, the error term we get from Theorem \ref{thm:grimmelt10.1} is
    \[
    O_{\e} \left( x^{\frac12+O(\eta)} r_0^\e r_2^{\frac12}\gcd(\widetilde{r}_2-\widetilde{r}_1,r_0^\infty)^{1+\e} \left( \left(\frac{|h(r_2-r_1)|}{r_0} \right)^\theta+\frac{x^\theta}{r_0^\theta}  \right) \right).
    \]
\end{proof}
\noindent The following theorem relaxes the condition $(h,r_1r_2)=1$ from Theorem \ref{thm:certaindivisorcoprime}.
\begin{thm} \label{thm:certaindivisor}
    Let $r_1,r_2 \in \mathbb{Z}^+$ be squarefree and distinct, and $h \in \mathbb{Z} \setminus \{0\}$. Let $r_0=(r_1,r_2)$, $\widetilde{r}_i=r_i/r_0$ for $i=1,2$, $x,\eta,\e \in \mathbb{R}^+$, and  define $w_1,w_2:\mathbb{R} \to \mathbb{R}$ smooth functions compactly supported on $[1/2,1]$, satisfying $w_1^{(j)}(x),w_2^{(j)}(x) \ll_j x^{j \eta}$ for all $j \geq 0$. Then for $|h| \ll x^{1-\e}$, we have
    \begin{align*}
        \sum_n w_1 \left( \frac{r_1n}{x} \right) &w_2\left( \frac{r_2n}{x} \right) d(r_1n+h) d(r_2n+h)=\textrm{Main Term}\\
        &+O_{\e} \left( x^{\frac12+O(\eta)} r_0^\e r_2^{\frac12+\e}\gcd(\widetilde{r}_2-\widetilde{r}_1,r_0^\infty)^{1+\e} \left( \left(\frac{|h(r_2-r_1)|}{r_0} \right)^\theta+\frac{x^\theta}{r_0^\theta}  \right) \right),
    \end{align*}
    where the main term is given by
    \[
    \int_\mathbb{R} w_1 \left( \frac{r_1 \xi}{x} \right) w_2 \left( \frac{r_2 \xi}{x} \right) P(\log(r_1 \xi+h),\log(r_2 \xi+h)) \d \xi,
    \]
    and $P(X,Y)$ is a quadratic polynomial depending only on $r_1,r_2,h$.
\end{thm}
\begin{proof}
    As in the proof of Theorem \ref{thm:certaindivisorcoprime}, we may arrange the sum to be of the form
    \begin{align} \label{eqn:acertaindivisorfinalfirsteq}
        \sum_{ad-bc=h(\widetilde{r}_2-\widetilde{r}_1)}w_1 \left( \frac{ad-h\widetilde{r}_2}{\widetilde{r}_2x} \right) w_2 \left( \frac{bc-h\widetilde{r}_1}{\widetilde{r}_1x} \right) \mathds{1}_{\widetilde{r}_2 \mid a}\mathds{1}_{\widetilde{r}_1 \mid c} \mathds{1}_{r_2 \mid a d-h\widetilde{r}_2}.
    \end{align}
    Let $s_0=\gcd(h,r_0)$ and $s_i=\gcd(h,\widetilde{r}_i)$ for $i=1,2$. Since $r_1$ and $r_2$ are squarefree, $s_0,s_1,s_2$ are also squarefree and $(s_0,s_1)=(s_0,s_2)=(s_1,s_2)=1$, so $s_0s_1s_2 \mid h$. From the indicator functions, note that $s_1 \mid c$ and $s_2 \mid a$. Also, from $s_0 \mid h$ and $r_0 \mid ad-h\widetilde{r}_2$ we have $s_0 \mid ad$. Thus, $s_0 \mid bc$ as well. Thus, $ad-bc=h(\widetilde{r}_2-\widetilde{r}_1)$ implies $(a/s_2)d/s_0-bc/s_0s_2=h(\widetilde{r}_2-\widetilde{r}_1)/s_0s_2$, which implies $s_0s_2 \mid bc$. Since $(s_1,s_2s_0)=1$, we also have $s_0s_2 \mid b(c/s_1)$. Similarly, we also have $s_0s_1 \mid (a/s_2)d$.\\ 
    
    For $s$ squarefree, the inclusion-exclusion principle gives
    \begin{align*}
        \mathds{1}_{s \mid ad}=\prod_{p \mid s} \left( \mathds{1}_{p \mid a}+\mathds{1}_{p \mid d}-\mathds{1}_{p \mid a} \mathds{1}_{p \mid d} \right)=\sum_{[u,v]=s} (-1)^{\omega((u,v))} \mathds{1}_{u \mid a} \mathds{1}_{v \mid d},
    \end{align*}
    and therefore \eqref{eqn:acertaindivisorfinalfirsteq} can be written as
    \begin{align}
        &\sum_{\frac{(a/s_2)d}{s_0s_1}-\frac{b(c/s_1)}{s_0s_2}=\frac{h}{s_0s_1s_2}(\widetilde{r}_2-\widetilde{r}_1)} w_1 \left( \frac{ad-h\widetilde{r}_2}{\widetilde{r}_2x} \right) w_2 \left( \frac{bc-h\widetilde{r}_1}{\widetilde{r}_1x} \right) \mathds{1}_{\widetilde{r}_2 \mid a}\mathds{1}_{\widetilde{r}_1 \mid c} \mathds{1}_{r_0 \widetilde{r}_2 \mid a d-h\widetilde{r}_2}\notag{}\\
        =&\sum_{\substack{[u_1,v_1]=s_1\\ [u_2,v_2]=s_2}} \sum_{\substack{[u_0,v_0]=s_0\\ [u_0',v_0']=s_0}}(-1)^{\omega((u_1,v_1))+\omega((u_2,v_2))+\omega((u_0,v_0))+\omega((u_0',v_0'))}\notag{}\\
        &\sum_{\frac{(a/s_2)d}{s_0s_1}-\frac{b(c/s_1)}{s_0s_2}=\frac{h}{s_0s_1s_2}(\widetilde{r}_2-\widetilde{r}_1)} w_1 \left( \frac{ad-h\widetilde{r}_2}{\widetilde{r}_2x} \right) w_2 \left( \frac{bc-h\widetilde{r}_1}{\widetilde{r}_1x} \right) \mathds{1}_{\widetilde{r}_2 \mid a}\mathds{1}_{\widetilde{r}_1 \mid c} \mathds{1}_{r_0 \widetilde{r}_2 \mid a d-h\widetilde{r}_2}\mathds{1}_{u_0u_1 \mid a} \mathds{1}_{v_0v_1 \mid d} \mathds{1}_{u_0'u_2 \mid b} \mathds{1}_{v_0'v_2 \mid c} \notag{} \\
        =&\sum_{\substack{[u_1,v_1]=s_1\\ [u_2,v_2]=s_2}} \sum_{\substack{[u_0,v_0]=s_0\\ [u_0',v_0']=s_0}}(-1)^{\omega((u_1,v_1))+\omega((u_2,v_2))+\omega((u_0,v_0))+\omega((u_0',v_0'))}\notag{}\\
        &\sum_{(u_0,v_0)(u_1,v_1)ad-(u_0',v_0')(u_2,v_2)bc=h'(\widetilde{r}_2-\widetilde{r}_1)} w_1 \left( \frac{(u_0,v_0)(u_1,v_1)ad-h'\widetilde{r}_2}{\widetilde{r}_2(x/s_0s_1s_2)} \right) w_2 \left( \frac{(u_0',v_0')(u_2,v_2)bc-h'\widetilde{r}_1}{\widetilde{r}_1(x/s_0s_1s_2)} \right)\notag{}\\
        &\hspace{20em}\mathds{1}_{\widetilde{r}_2' \mid a} \mathds{1}_{\widetilde{r}_1' \mid c} \mathds{1}_{r_0' \widetilde{r}_2' \mid  (u_0,v_0)(u_1,v_1)ad-s_2 h'\widetilde{r}_2'}, \label{eqn:presufficefinalcertaindivisor}
    \end{align}
    where $h'=h/s_0s_1s_2$, $r_0'=r_0/s_0$, $\widetilde{r}_1'=\widetilde{r}_1/s_1$ and $\widetilde{r}_2'=\widetilde{r}_2/s_2$, and we made the substitution
    \[
    a \mapsto \frac{a}{s_2u_0u_1}, \quad d \mapsto \frac{d}{v_0v_1}, \quad b \mapsto \frac{b}{u_0'u_2}, \quad c \mapsto \frac{c}{s_1v_0'v_2}.
    \]
    Making another substitution, \eqref{eqn:presufficefinalcertaindivisor} equals
    \begin{align}
            &\sum_{\substack{[u_1,v_1]=s_1\\ [u_2,v_2]=s_2}} \sum_{\substack{[u_0,v_0]=s_0\\ [u_0',v_0']=s_0}}(-1)^{\omega((u_1,v_1))+\omega((u_2,v_2))+\omega((u_0,v_0))+\omega((u_0',v_0'))}\sum_{ad-bc=h'(\widetilde{r}_2-\widetilde{r}_1)} w_1 \left( \frac{ad-h'\widetilde{r}_2}{\widetilde{r}_2(x/s_0s_1s_2)} \right)\notag{}\\
        &\hspace{1em} w_2 \left( \frac{bc-h'\widetilde{r}_1}{\widetilde{r}_1(x/s_0s_1s_2)} \right) \mathds{1}_{\widetilde{r}_2' \mid a} \mathds{1}_{\widetilde{r}_1' \mid c} \mathds{1}_{r_0' \widetilde{r}_2' \mid a d-s_2h'\widetilde{r}_2'}\mathds{1}_{(u_0,v_0)(u_1,v_1) \mid a}\mathds{1}_{(u_0',v_0')(u_2,v_2) \mid c}\notag{}\\
        =&\sum_{\substack{[u_1,v_1]=s_1\\ [u_2,v_2]=s_2}} \sum_{\substack{[u_0,v_0]=s_0\\ [u_0',v_0']=s_0}}(-1)^{\omega((u_1,v_1))+\omega((u_2,v_2))+\omega((u_0,v_0))+\omega((u_0',v_0'))}\notag{}\\
        &\hspace{5em}\sum_{ad-bc=h'(\widetilde{r}_2-\widetilde{r}_1)} w_1 \left( \frac{ad-h'\widetilde{r}_2}{\widetilde{r}_2(x/s_0s_1s_2)} \right) w_2 \left( \frac{bc-h'\widetilde{r}_1}{\widetilde{r}_1(x/s_0s_1s_2)} \right) \mathds{1}_{\widetilde{r}_2'' \mid a} \mathds{1}_{\widetilde{r}_1'' \mid c} \mathds{1}_{r_0' \widetilde{r}_2'' \mid  a d-s_2\ell h'\widetilde{r}_2''}, \label{eqn:sufficefinalcertaindivisor}
    \end{align}
    where $\widetilde{r}_1''=\widetilde{r}_1'(u_0',v_0')(u_2,v_2)$, $\widetilde{r}_2''=\widetilde{r}_2'(u_0,v_0)(u_1,v_1)$, and $\ell=\overline{(u_0,v_0)(u_1,v_1)}$ denotes the inverse of $(u_0,v_0)(u_1,v_1) \Mod{r_0\widetilde{r}_2''/(u_0,v_0)(u_1,v_1)}$, which exists due to the coprimality of $r_0\widetilde{r}_2''/(u_0,v_0)(u_1,v_1)$ and $(u_0,v_0)(u_1,v_1)$. To apply an analogous argument as in \eqref{eqn:referforanalogouscertaindivisorcoprime}, we need $\gcd(h',r_0'r_1''r_2'')=1$, which is not necessarily true. Therefore, we aim to remove the possible common prime factors of $h'$ and the
coefficients multiplying $ad$ and $bc$ in \eqref{eqn:presufficefinalcertaindivisor}. We describe an iterative
procedure. Let
\[
F_0=(u_0,v_0)(u_1,v_1), \quad
G_0=(u_0',v_0')(u_2,v_2), \quad
H_0=h', \quad E_0=F_0, \quad E_0'=G_0
\]
so that the inner sum of \eqref{eqn:presufficefinalcertaindivisor} is of the shape
\[
\sum_{F_0ad-G_0bc=H_0(\widetilde r_2-\widetilde r_1)} w_1\left(
\frac{E_0ad-H_0\widetilde r_2}
{\widetilde r_2(x/s_0s_1s_2)}
\right)
w_2\left(
\frac{E_0'bc-H_0\widetilde r_1}
{\widetilde r_1(x/s_0s_1s_2)}
\right)
\mathds 1_{\widetilde r_2'\mid a}
\mathds 1_{\widetilde r_1'\mid c}
\mathds 1_{r_0'\widetilde r_2'
\mid E_0ad-s_2 H_0\widetilde r_2'}.
\]
Notice that $(H_0,r_0'\widetilde r_1\widetilde r_2)=1,$ although $H_0$ need not be coprime to $F_0G_0$. Suppose inductively that after $j$ steps we have obtained a finite linear
combination of sums of the form
\[
\sum_{F_jad-G_jbc=H_j(\widetilde r_2-\widetilde r_1)} w_1\left(
\frac{E_jad-H_0\widetilde r_2}
{\widetilde r_2(x/s_0s_1s_2)}
\right)
w_2\left(
\frac{E_j'bc-H_0\widetilde r_1}
{\widetilde r_1(x/s_0s_1s_2)}
\right)
\mathds 1_{\widetilde r_2'\mid a}
\mathds 1_{\widetilde r_1'\mid c}
\mathds 1_{r_0'\widetilde r_2'
\mid E_jad-s_2 H_0\widetilde r_2'},
\]
where $(H_j,r_0'\widetilde r_1\widetilde r_2)=1$, $F_j,G_j$ are squarefree, $F_j \mid E_j$, $G_j \mid E_j'$, $\operatorname{rad}(E_jE_j') \mid s_0s_1s_2$ and
\begin{align}
    \frac{F_j}{E_j}=\frac{G_j}{E_j'}=\frac{H_j}{H_0}. \label{eqn:iterativeinvariant}
\end{align}
Let $q_j=(H_j,F_jG_j)$. We have four possible actions:
\begin{itemize}
\item Action 0: If $q_j=1$, we are done.
\item Action 1: Suppose $q_j>1$. If $q_j$ has a common factor with both $F_j$ and $G_j$, then this factor also divides the left-hand side of
\begin{align} \label{eqn:jthiterationsummationcondition}
    F_jad-G_jbc &=H_j(\widetilde r_2-\widetilde r_1),
\end{align}
and also $H_j$. Dividing the entire equation by $(H_j,F_j,G_j)$ produces a
sum of the same shape with $H_j$ replaced by a proper divisor, say $H_{j+1}$, and let $F_{j+1}=F_j/(H_j,F_j,G_j)$ and $G_{j+1}=G_j/(H_j,F_j,G_j)$. We keep $E_{j+1}=E_j$ and $E_{j+1}'=E_j'$ unchanged. Clearly, \eqref{eqn:iterativeinvariant} is true for $j+1$.
\item Action 2: Suppose now that $t_j=(H_j,F_j)>1$ and $(H_j,F_j,G_j)=1$. Since $(F_j,G_j,H_j)=1$, we have $(t_j,G_j)=1$. Reducing \eqref{eqn:jthiterationsummationcondition} modulo $t_j$ gives $t_j\mid bc$. Applying
\[
\mathds 1_{t_j\mid bc}
=
\sum_{[\alpha,\beta]=t_j}
(-1)^{\omega((\alpha,\beta))}
\mathds 1_{\alpha\mid b}
\mathds 1_{\beta\mid c},
\]
and making the substitution $b \mapsto \alpha b$, and $c \mapsto \beta c$, we obtain a finite linear combination of sums of the same shape
\begin{align*}
    \sum_{F_{j+1}ad-G_{j+1}bc=H_{j+1}(\widetilde r_2-\widetilde r_1)} w_1\left(
\frac{E_{j+1}ad-H_0\widetilde r_2}
{\widetilde r_2(x/s_0s_1s_2)}
\right)
w_2\left(
\frac{E_{j+1}'bc-H_0\widetilde r_1}
{\widetilde r_1(x/s_0s_1s_2)}
\right)
\mathds 1_{\widetilde r_2'\mid a}
\mathds 1_{\widetilde r_1'\mid c}
\mathds 1_{r_0'\widetilde r_2'
\mid E_{j+1}ad-s_2 H_0\widetilde r_2'}
\end{align*}
with
\[
F_{j+1}=\frac{F_j}{t_j},
\quad
G_{j+1}=G_j(\alpha,\beta),
\quad
H_{j+1}=\frac{H_j}{t_j}, \quad E_{j+1}=E_j, \quad E_{j+1}'=E_j'\alpha \beta.
\]
Since $(t_j,r_0'\widetilde r_1\widetilde r_2)=1$, the divisibility conditions are
unchanged. Also, since $F_j$ is squarefree, $t_j$ is also squarefree and coprime to $G_j$, and hence $(\alpha,\beta)$ is squarefree and coprime to $G_j$, and so both $F_{j+1}$ and $G_{j+1}$ are still squarefree. Moreover, $(H_{j+1},F_{j+1})=1$, $F_{j+1} \mid E_{j+1}$, $G_{j+1} \mid E_{j+1}'$, $\operatorname{rad}(E_{j+1}E_{j+1}') \mid s_0s_1s_2$ and
\begin{align*}
    \frac{F_{j+1}/E_{j+1}}{F_j/E_j}=\frac{G_{j+1}/E_{j+1}'}{G_j/E_j'}=\frac{H_{j+1}/H_0}{H_j/H_0}=\frac{1}{t_j},
\end{align*}
and so \eqref{eqn:iterativeinvariant} holds for $j+1$.
\item Action 3: Similarly, if $t_j'=(H_j,G_j)>1$ and $(H_j,F_j,G_j)=1$, then $(t_j',F_j)=1$, and reducing \eqref{eqn:jthiterationsummationcondition} modulo $t_j'$ shows that $t_j'\mid ad$. Applying the analogous inclusion-exclusion identity for $\mathds 1_{t_j'\mid ad}$ yields a finite linear combination of sums of the same shape
\begin{align*}
    \sum_{F_{j+1}ad-G_{j+1}bc=H_{j+1}(\widetilde r_2-\widetilde r_1)} w_1\left(
\frac{E_{j+1}ad-H_0\widetilde r_2}
{\widetilde r_2(x/s_0s_1s_2)}
\right)
w_2\left(
\frac{E_{j+1}'bc-H_0\widetilde r_1}
{\widetilde r_1(x/s_0s_1s_2)}
\right)
\mathds 1_{\widetilde r_2'\mid a}
\mathds 1_{\widetilde r_1'\mid c}
\mathds 1_{r_0'\widetilde r_2'
\mid E_{j+1}ad-s_2 H_0\widetilde r_2'}
\end{align*}
with
\[
F_{j+1}=F_j(\alpha,\beta),
\quad
G_{j+1}=\frac{G_j}{t_j'},
\quad
H_{j+1}=\frac{H_j}{t_j'}, \quad E_{j+1}=E_j\alpha \beta, \quad E_{j+1}'=E_j'.
\]
Also, both $F_{j+1}$ and $G_{j+1}$ are still squarefree, $(H_{j+1},G_{j+1})=1$, $F_{j+1} \mid E_{j+1}$, $G_{j+1} \mid E_{j+1}'$, $\operatorname{rad}(E_{j+1}E_{j+1}') \mid s_0s_1s_2$ and
\begin{align*}
    \frac{F_{j+1}/E_{j+1}}{F_j/E_j}=\frac{G_{j+1}/E_{j+1}'}{G_j/E_j'}=\frac{H_{j+1}/H_0}{H_j/H_0}=\frac{1}{t_j'},
\end{align*}
and so \eqref{eqn:iterativeinvariant} holds for $j+1$.
\end{itemize}
Observe that we only need to do Action 1 once, then we alternate between Action 2 and Action 3 until Action 0 is satisfied, and we terminate the iteration.
In every non-terminal step, $H_j$ is replaced by a proper divisor of itself. Therefore, the procedure terminates after at most $\Omega(h')$ steps, where $\Omega(\cdot)$ denotes the number of prime factors counted with multiplicity. The total number of resulting terms is therefore
\[
\ll_\varepsilon d(h')^{O(1)}
\ll_\varepsilon x^\varepsilon,
\]
since $h'\ll x^{1-\varepsilon}$.
For each terminal term, we obtain parameters $E$, $E'$, $F$, $G$, $H$ with $F,G$ squarefree, $\operatorname{rad}(EE') \mid s_0s_1s_2$ and
\[
(H,F G)=1,
\quad
(H,r_0'\widetilde r_1\widetilde r_2)=1, \quad F \mid E, \quad G \mid E'.
\]
Finally, letting $\widetilde r_2'''=\widetilde r_2'F$, $\widetilde r_1'''=\widetilde r_1'G$, and $\eta_{E,E',F,G,H} \in \{1,-1\}$, we arrive at sums of the form

\begin{align}
&\sum_{E,E',F,G,H} \eta_{E,E',F,G} \sum_{Fad-Gbc=H(\widetilde r_2-\widetilde r_1)} w_1\left(
\frac{Ead-H_0\widetilde r_2}
{\widetilde r_2(x/s_0s_1s_2)}
\right)
w_2\left(
\frac{E'bc-H_0\widetilde r_1}
{\widetilde r_1(x/s_0s_1s_2)}
\right)
\mathds 1_{\widetilde r_2'\mid a}
\mathds 1_{\widetilde r_1'\mid c}
\mathds 1_{r_0'\widetilde r_2'
\mid Ead-s_2 H_0\widetilde r_2'} \notag{}\\
&=\sum_{E,E',F,G,H} \eta_{E,E',F,G} \sum_{ad-bc=H(\widetilde r_2-\widetilde r_1)} w_1\left(
\frac{Ead-H_0\widetilde r_2F}
{\widetilde r_2(xF/s_0s_1s_2)}
\right)
w_2\left(
\frac{E'bc-H_0\widetilde r_1G}
{\widetilde r_1(xG/s_0s_1s_2)}
\right)\notag{}
\\
&\hspace{18.5em}\mathds 1_{\widetilde r_2'F\mid a}
\mathds 1_{\widetilde r_1'G\mid c}
\mathds 1_{r_0'\widetilde r_2'F
\mid Ead-s_2 H_0\widetilde r_2'F} \notag{} \notag{}\\
&=\sum_{E,E',F,G,H} \eta_{E,E',F,G} \sum_{ad-bc=H(\widetilde r_2-\widetilde r_1)} w_1\left(
\frac{ad-s_2H_0\widetilde r_2'''/E}
{\widetilde r_2(xF/Es_0s_1s_2)}
\right)
w_2\left(
\frac{bc-s_1H_0\widetilde r_1'''/E'}
{\widetilde r_1(xG/E's_0s_1s_2)}
\right)\notag{}\\
&\hspace{18.5em}\mathds 1_{\widetilde r_2'''\mid a}
\mathds 1_{\widetilde r_1'''\mid c}
\mathds 1_{{r_0'\widetilde r_2'}
\mid (E/F)ad-{s_2 H_0\widetilde r_2'}}. \label{eqn:presufficefinalcertaindivisorafteriteration}
\end{align}
We rewrite the last congruence condition. First, note $\operatorname{rad}(E/F) \mid s_0s_1s_2$ and hence is coprime to $r_0'\widetilde{r}_2'$, so let $\lambda=\overline{E/F}$ be its inverse mod $r_0'\widetilde{r}_2'$. Then,
\begin{align*}
    \mathds{1}_{\widetilde{r}_2''' \mid a}\mathds 1_{ad \equiv {s_2\lambda H_0\widetilde r_2'} \Mod{r_0'\widetilde r_2'}} &=\mathds{1}_{\widetilde{r}_2''' \mid a} \mathds{1}_{ad \equiv 0 \Mod{F}} 1_{ad \equiv {s_2\lambda H_0\widetilde r_2'} \Mod{r_0'\widetilde r_2'}}\\
    &= \mathds{1}_{\widetilde{r}_2''' \mid a}\mathds{1}_{ad \equiv 0 \Mod{F}}\mathds 1_{ad \equiv {s_2\lambda \overline{F}H_0\widetilde r_2'''} \Mod{r_0'\widetilde r_2'}}\\
    &=\mathds{1}_{\widetilde{r}_2''' \mid a}\mathds 1_{ad \equiv {s_2\lambda \overline{F} H_0\widetilde r_2'''} \Mod{r_0'\widetilde r_2'''}},
\end{align*}
where $\overline{F}$ denotes the inverse of $F$ mod $r_0'r_2'$, which exists since they are coprime. Thus, \eqref{eqn:presufficefinalcertaindivisorafteriteration} equals
\begin{align}
&\sum_{E,E',F,G,H} \eta_{E,E',F,G} \sum_{ad-bc=H(\widetilde r_2-\widetilde r_1)} w_1\left(
\frac{ad-s_2H_0\widetilde r_2'''/E}
{\widetilde r_2(xF/Es_0s_1s_2)}
\right)
w_2\left(
\frac{bc-s_1H_0\widetilde r_1'''/E'}
{\widetilde r_1(xG/E's_0s_1s_2)}
\right)\notag{}\\
&\hspace{17.5em}
\mathds 1_{\widetilde r_2'''\mid a}
\mathds 1_{\widetilde r_1'''\mid c}
\mathds 1_{r_0'\widetilde r_2'''
\mid ad-s_2\lambda \overline{F} H_0\widetilde r_2'''}, \notag{}\\
&\sum_{E,E',F,G,H} \eta_{E,E',F,G} \sum_{ad-bc=H(\widetilde r_2-\widetilde r_1)} w_1\left(
\frac{ad-H\widetilde r_2}
{\widetilde r_2(Hx/h)}
\right)
w_2\left(
\frac{bc-H\widetilde r_1}
{\widetilde r_1(Hx/h)}
\right)\notag{}\\
&\hspace{17.5em}
\mathds 1_{\widetilde r_2'''\mid a}
\mathds 1_{\widetilde r_1'''\mid c}
\mathds 1_{r_0'\widetilde r_2'''
\mid ad-s_2\lambda \overline{F} H_0\widetilde r_2'''}, \label{eqn:sufficefinalcertaindivisorafteriteration}
\end{align}
with $(H,r_0'\widetilde{r}_1'''\widetilde{r}_2''')=1$ and $(s_2\lambda \overline{F} H_0,r_0')=1$. The outer summation ranges over all choices arising from the initial
inclusion-exclusion together with the above iterative procedure, and contains
at most $O_\varepsilon(x^\varepsilon\widetilde r_2^\varepsilon)$ terms. Then, an analogous argument to \eqref{eqn:referforanalogouscertaindivisorcoprime} suffices to evaluate \eqref{eqn:sufficefinalcertaindivisorafteriteration}, obtaining the required error term. The main term is of the form
\[
\sum_{E,E',F,G,H}C_{E,E',F,G,H}
\int_{\mathbb R}
w_1\left(\frac{v}{\widetilde r_2 xH/h}\right)
w_2\left(\frac{v}{\widetilde r_1 xH/h}\right)
Q_{E,E',F,G,H}\left(
\log\frac{v+H\widetilde r_2}{\widetilde r_2'''},\,
\log\frac{v+H\widetilde r_1}{\widetilde r_1'''}
\right)\,dv,
\]
where $Q_{E,E',F,G,H}(X,Y)$ is a quadratic polynomial and $C_{E,E',F,G,H}$ is a constant depending only on the corresponding  terminal data. Using \eqref{eqn:iterativeinvariant} together with
\[
\widetilde r_2'''=\widetilde r_2'F=\frac{\widetilde r_2F}{s_2},
\qquad
\widetilde r_1'''=\widetilde r_1'G=\frac{\widetilde r_1G}{s_1},
\]
we make the change of variables
\[
v=\frac{r_1r_2}{r_0}\frac{H}{h}\,\xi.
\]
Then
\[
\frac{v}{\widetilde r_2 xH/h}=\frac{r_1\xi}{x},
\qquad
\frac{v}{\widetilde r_1 xH/h}=\frac{r_2\xi}{x},
\]
while
\[
\log\frac{v+H\widetilde r_2}{\widetilde r_2'''}
=
\log(r_1\xi+h)+c_1,
\qquad
\log\frac{v+H\widetilde r_1}{\widetilde r_1'''}
=
\log(r_2\xi+h)+c_2,
\]
where $c_1=\log({Hs_2}/{Fh})$ and $c_2=\log({Hs_1}/{Gh})$ are constants depending only on the corresponding terminal data.
Hence this contribution is of the form
\[
\sum_{E,E',F,G,H}\int_{\mathbb R}
w_1\left(\frac{r_1\xi}{x}\right)
w_2\left(\frac{r_2\xi}{x}\right)
P_{E,E',F,G,H}
\bigl(\log(r_1\xi+h),\log(r_2\xi+h)\bigr)\,d\xi,
\]
where $P_{E,E',F,G,H}(X,Y)$ is again a quadratic polynomial. This equals
\[
\int_{\mathbb R}
w_1\left(\frac{r_1\xi}{x}\right)
w_2\left(\frac{r_2\xi}{x}\right)
P(\log(r_1\xi+h),\log(r_2\xi+h))\,d\xi,
\]
for some quadratic polynomial $P(X,Y)$ depending only on $h,r_1,r_2$, as required.
\end{proof}

\section{Shifted Convolution of Generalised Divisor Functions}
\noindent In this section, we prove Theorem \ref{thm:generaliseddivisor}.
\setcounter{section}{1}
\setcounter{thm}{1}
\begin{thm}\label{thm:generaliseddivisor}
    Let $w:[1/2,1] \to [0,\infty)$ be smooth and compactly supported, $\e>0$, and $x \in \mathbb{R}^+$ sufficiently large in terms of $\e$. Let $k \ge 4$ and $h \in \mathbb{Z} \backslash \{0\}$. Then for $0 \leq \delta \leq \frac{1}{16}$ and $|h| \ll x^{1-\e}$, we have
    \begin{align*}
        \sum_n w \left( \frac{n}{x} \right) &d_k(n) d(n + h)= \ \int_\mathbb{R} w \left( \frac{\xi}{x} \right) P_{k,h}(\log x,\log \xi,\log(\xi+h)) \d \xi\\
        &+O_{k,w,\e,\delta}\left(x^{1-\delta+2 \delta \theta+\e}\left( 1+\frac{|h|^{\frac{1}{4}}}{x^{\frac14-\frac12 \delta}} \right)+x^{1-\delta+\frac{\theta}{2}+\e} \left(1+\frac{|h|^{\frac{\theta}{2}}}{x^{\frac{\theta}{3}+\frac{2 \delta}{3} \theta}}\right) \right),
    \end{align*}
    where $P_{k,h}$ is a polynomial of degree $k$ depending only on $k$ and $h$, and the implied constant depends on $k,w,\e,\delta$.
\end{thm}
\noindent The proof of Theorem \ref{thm:generaliseddivisor} relies on the following lemma.
\setcounter{section}{5}
\setcounter{thm}{0}
\begin{lem} \label{lem:remainderbound}
    Let $w:[1/2,1] \to [0,\infty)$ be a smooth and compactly supported function and $x \in \mathbb{R}^+$. Let $k \ge 4$, $h \in \mathbb{N}$, and let $v_1,\ldots,v_k$ be smooth and compactly supported functions with $\operatorname{supp} v_i \asymp A_i$ and $v_i^{(\nu)} \ll A_i^{-\nu}$ for $\nu \geq 0$. Then,
    \[
    \sum_{a_1,\ldots,a_k} w \left( \frac{a_1 \cdots a_k}{x} \right) v_1(a_1) \cdots v_k(a_k) d(a_1 \cdots a_k+h)=M_{v_1}+R_{v_1},
    \]
    where $M_{v_1}$ is the main term 
    $$M_{v_1} := \int_\mathbb{R} w \left( \frac{\xi}{x} \right) \sum_{\substack{a_2,\ldots,a_k\\ d \mid a_2 \cdots a_k}} \frac{v_2(a_2) \cdots v_k(a_k)}{a_2 \cdots a_k} v_1 \left( \frac{\xi}{a_2 \cdots a_k} \right) \lambda_{h,d}(\xi+h) \d \xi,$$
    with
    \[
    \lambda_{h,d}(\xi)=\frac{c_d(h)}{d}(\log \xi+2 \gamma-2\log d), \quad c_d(h) := \sum_{\substack{a \Mod{d}\\ (a,d)=1}} e\left( \frac{ah}{d} \right),
    \]
    and the following bounds for the error term $R_{v_1}$ for $h \ll x^{1-\e}$,
    \begin{align}
        R_{v_1} &\ll \frac{x^{\frac{3}{2}+\e}}{A_1^{\frac{3}{2}}}, \label{rem1}\\
        R_{v_1} &\ll \frac{x^{\frac{3}{2}+\e}}{A_1A_2} \left( 1+\frac{(A_1A_2)^{2 \theta}}{x^\theta} \right) \left( 1+\frac{A_2^{\frac{1}{2}}}{A_1^{\frac{1}{2}}} \right) \left( 1+|h|^{\frac{1}{4}} \frac{A_1^{\frac{1}{2}}A_2^{\frac{1}{2}}}{x^{\frac{1}{2}}} \right), \label{rem2}\\
        R_{v_1} &\ll \frac{x^{1+\e}}{A^{\frac12}}+A^{\frac34}x^{\frac34+\e} \left( |h|^{\frac{\theta}{2}} A^{\frac{\theta}{2}}+x^{\frac\theta2} \right),  \label{rem3}
    \end{align}
    where $A=\prod_{i \in I}A_i$ for any non-empty index set $I \subseteq\{1,\ldots,k\}$. The implied constants depend only on $k$, and $w,v_1,\ldots,v_k,\e$.
\end{lem}
\begin{proof}
    The first two bounds are from Topacogullari \cite[Lemma 2.1]{topacogullari2018shifted}. Let
    \[
    \Psi_{v_1,\ldots,v_k}=\sum_{a_1,\ldots,a_k} w \left( \frac{a_1 \cdots a_k}{x} \right) v_1(a_1) \cdots v_k(a_k) d(a_1 \cdots a_k+h).
    \]
    Our treatment of the third bound is a modification of Topacogullari \cite{topacogullari2018shifted}. First we rename variables such that $I=\{1,2,\ldots,|I|\}$. Write
    \[
    \Psi_{v_1,\ldots,v_k}=\sum_b \delta_I(b) \Phi_I(b),
    \]
    where
    \begin{align*}
        \Phi_I(b)&=\sum_{a_i, \, i \in I} d\left( b \prod_{i \in I} a_i+h \right) w \left( \frac{b}{x} \prod_{i \in I}a_i \right) \prod_{i \in I} v_i (a_i),\\
        \delta_I(b)&=\sum_{\substack{a_{|I|+1},\ldots,a_k\\ b=a_{|I|+1} \cdots a_k}} v_{|I|+1}(a_{|I|+1}) \cdots v_k(a_k).
    \end{align*}
    Let $v^{sq}:(0,\infty) \to [0,\infty)$ be a smooth function compactly supported on $[1,2]$, such that
    \[
    \sum_{j \in \mathbb{Z}} v^{sq} \left( \frac{u}{2^j} \right)=1
    \]
    for any $u \in \mathbb{R}^+$, and let $v_j^{sq}(u)=v^{sq}(u/2^j)$. We decompose each $a_i$ according to its square part and the pattern of its squarefree prime divisors. More precisely, there is a bijection between tuples of positive integers $(a_i)_{i\in I}$ and collections of variables $(c_i)_{i\in I}$, $ (a_i')_{i\in I}$, $(g_S)_{\substack{S\subseteq I\\|S|\ge2}},$
with $\prod_i a_i'$ and $\prod_S g_S$ squarefree and coprime to each other, and the bijection is given by
\[
a_i
=
c_i^2\,a_i'
\prod_{\substack{S\subseteq I\\ i\in S\\ |S|\ge2}}
g_S,
\qquad i\in I.
\]
Here:
\begin{itemize}
    \item $c_i^2$ is the largest square dividing $a_i$;
    \item $a_i'$ is the product of the prime divisors of $a_i/c_i^2$ that divide no other $a_j/c_j^2$;
    \item for each subset $S\subseteq I$ with $|S|\ge2$, $g_S$ is the product of the prime divisors that divide $a_i/c_i^2$ if and only if $i\in S$.
\end{itemize}
Equivalently, every prime divisor of the tuple $(a_i/c_i^2)_{i\in I}$ is assigned uniquely either to one of the variables $a_i'$ (if it divides exactly one component) or to one of the variables $g_S$ (if it divides precisely those components indexed by $S$). In particular, the variables $\prod_i a_i'$ and $\prod_S g_S$ are squarefree and coprime
\begin{align*}
    \mu^2 \left( \prod_{i \in I} a_i' \right)=\mu^2 \Bigg( \prod_{\substack{S \subseteq I\\ |S| \ge 2}} g_S \Bigg)=\gcd \Bigg( \prod_{i \in I} a_i',
\prod_{\substack{S \subseteq I\\ |S| \ge 2}} g_S \Bigg)=1.
\end{align*}
    Inserting this into $\Psi_{v_1,\ldots,v_k}$, we have
    \begin{align*}
        \Psi_{v_1,\ldots,v_k}
        &= \sum_b \delta_I(b) \sum_{\substack{a_i,\ i \in I}} d \left( b \prod_{i \in I} a_i+h \right) w \left( \frac{b}{x} \prod_{i \in I}a_i \right) \prod_{i \in I} v_i(a_i) \\
    &= \sum_{j_r,r \in I} \sum_b \delta_I(b) \sum_{\substack{c_i,\ i \in I}} \sum_{\substack{g_S,\ S \subseteq I \\ |S| \ge 2 \\ \mu^2(\prod g_S) = 1}} \sum_{\substack{a'_i,\ i \in I \\ \mu^2(\prod a'_i) = 1 \\ (\prod a'_i,\, \prod g_S) = 1}} d \Bigg( b \prod_{i \in I} c_i^2 a'_i \prod_{\substack{S \subseteq I \\ |S| \ge 2}} g_S^{|S|} + h \Bigg) \\
    &\hspace{10em} \times w \Bigg( \frac{b}{x} \prod_{i \in I} c_i^2 a'_i \prod_{\substack{S \subseteq I \\ |S| \ge 2}} g_S^{|S|} \Bigg)\prod_{i \in I} v_{j_i}^{sq} \Bigg( c_i^2\prod_{\substack{S \subseteq I \\i \in S\\ |S| \ge 2}} g_S  \Bigg) v_i \Bigg( c_i^2 a'_i \prod_{\substack{S \subseteq I \\ i \in S \\ |S| \ge 2}} g_S \Bigg)\\
        &= \sum_{j_r,r \in I} \sum_{\substack{c_i,\ i \in I}} \sum_{\substack{g_S,\ S \subseteq I \\ |S| \ge 2 \\ \mu^2(\prod g_S) = 1}} \prod_{r \in I}v_{j_r}^{sq} \Bigg(c_r^2\prod_{\substack{S \subseteq I \\r \in S\\ |S| \ge 2}} g_S \Bigg)\sum_f \underbrace{\delta_I \left( \frac{f}{\prod_{i \in I}c_i^2\prod_{\substack{S \subseteq I \\ |S| \ge 2}} g_S^{|S|}} \right)}_{=: \delta_c'(f)}\\       
        &\hspace{13em}\underbrace{\sum_{\substack{a'_i,\ i \in I \\ \mu^2(\prod a'_i) = 1 \\ (\prod a'_i,\, \prod g_S) = 1}} d \Bigg( f\prod_{i \in I} a'_i + h \Bigg) w \Bigg( \frac{f}{x} \prod_{i \in I} a'_i\Bigg)  \prod_{i \in I} v_i \Bigg( c_i^2 a'_i \prod_{\substack{S \subseteq I \\ i \in S \\ |S| \ge 2}} g_S \Bigg)}_{=: \Phi_c'(f)}\\
        &= \sum_{j_r,r \in I}\sum_{\substack{c_i,\ i \in I}} \sum_{\substack{g_S,\ S \subseteq I \\ |S| \ge 2 \\ \mu^2(\prod g_S) = 1}}\prod_{r \in I} v_{j_r}^{sq} \Bigg( c_r^2\prod_{\substack{S \subseteq I \\r \in S\\ |S| \ge 2}} g_S \Bigg) \sum_{f} \delta_c'(f) \Phi_c'(f),
    \end{align*}
    where the sums over $f$ run over $f \equiv 0 \Mod {q}$ and $q := \prod_{i \in I}c_i^2\prod_{\substack{S \subseteq I \\ |S| \ge 2}} g_S^{|S|} \asymp \prod_{i \in I} C_i$.\\
    We let $\Psi_{c}' := \sum_f \delta_c'(f) \Phi_c'(f)$. Set $A'=\prod_{i \notin I}A_i\asymp x/\prod_{i \in I} A_i=x/A$. Let $C_{j_r}=2^{j_r}$, $\widetilde{A}_{j,i}=A_i/C_{j_i}$, and $\widetilde{A}_j=\prod_{i \in I} \widetilde{A}_{j,i}=A/\prod_{i \in I} C_{j_i}$. Observe $\widetilde{A}_j \ll A$. The main term of $\Phi_c'(f)$ is $\Phi_{c0}'(f)$, where
    \begin{align*}
        \Phi_{c0}'(f)=\frac{1}{fa_2 \cdots a_{|I|}} \sum_{\substack{a_2,\ldots,a_{|I|}\\ \prod_{i=2}^{|I|} a_i \text{ squarefree}\\ (\prod_{i=2}^{|I|} a_i, \, \prod_{\{1\} \subsetneq S \subseteq I} g_S)=1}} \int_\mathbb{R} &\Delta_\delta(\xi+h) w \left( \frac{\xi-h}{x} \right) v_1 \Bigg( \frac{c_1^2(\xi-h)}{fa_2 \cdots a_{|I|}} \prod_{\substack{\{1\} \subsetneq S \subseteq I }} g_S\Bigg) \\
        &\prod_{i=2}^{|I|} v_i \Bigg(c_i^2a_i \prod_{\substack{\{i\} \subsetneq S \subseteq I }} g_S\Bigg) \sum_{d \mid a_2 \cdots a_{|I|}f} \frac{\mu(d)^2c_d(h)}{d^{1+\delta}} \d \xi,
    \end{align*}
    and
    $\Delta_\delta(\xi)$ is the operator defined by
    \[
    \Delta_\delta(\xi)=\left( \log \xi + 2 \gamma + 2 \frac{\partial}{\partial \delta} \right) \Bigr|_{\delta=0}.
    \]
    We insert $\Phi_{c0}'(f)$ manually by
    \[
    \Psi_{c}'=\sum_{f \equiv 0 \Mod{q}} \delta_c'(f) \Phi_{c0}'(f)-\sum_{f \equiv 0 \Mod{q}} \delta_c'(f)(\Phi_{c0}'(f)-\Phi_c'(f)).
    \]
    Summing $\Phi_{c0}'(f)$ against $\delta_{c}'(f)$, then summing over $j$, the variables $(c_i)_{i\in I}$, and the variables $(g_S)_{S\subseteq I,\ |S|\ge2}$, we recover the main term $M_{v_1}$. Indeed, using the partition of unity
\[
\sum_{j\in\mathbb Z} v^{\mathrm{sq}}\left(\frac{u}{2^j}\right)=1,
\]
together with the bijection
\[
(a_i)_{i\in I}
\longleftrightarrow
\bigl((c_i)_{i\in I},(a_i')_{i\in I},(g_S)_{S\subseteq I,\ |S|\ge2}\bigr),
\]
the sums over $j$, $(c_i)$, and $(g_S)$ reconstruct the original summation over the variables $(a_i)_{i\in I}$. After unfolding the definitions of $\delta_{c}'(f)$ and $\Phi_{c0}'(f)$, and using the definition of $\lambda_{h,d}$, the resulting expression is the main term $M_{v_1}$ from Lemma \ref{lem:remainderbound}. Now let
$$R_{j}'=\sum_{f \equiv 0 \Mod{q}} \delta_c'(f)(\Phi_{c0}'(f)-\Phi_c'(f)).$$ We use Cauchy-Schwarz to get
    \[
    R_{j}' \leq \left( \sum_{\substack{f \asymp A'\prod_{i \in I}C_{j_i}\\f \equiv 0 \Mod{q}}} |\delta_c'(f)|^2 \right)^{1/2} \left( \sum_{f \equiv 0 \Mod{q}} |\Phi_{c0}'(f)-\Phi_c'(f)|^2 \right)^{1/2}.
    \]
    The first term can be estimated trivially by
    \[
    \sum_{\substack{f \asymp A'\prod_{i \in I}C_{j_i}\\f \equiv 0 \Mod{q}}} |\delta_c'(f)|^2 \ll \frac{x^\e A'\prod_{i \in I}C_{j_i}}{q} \ll \frac{x^{1+\e}}{q\widetilde{A}_j},
    \]
    and for the other factor we expand the square and write
    \[
    \sum_{f \equiv 0 \Mod{q}} |\Phi_{c0}'(f)-\Phi_c'(f)|^2=\Sigma_1'-2\Sigma_2'+\Sigma_3',
    \]
    where
    \[
    \Sigma_1'=\sum_{f \equiv 0 \Mod{q}} \Phi_{c0}'(f)^2, \quad \Sigma_2'=\sum_{f \equiv 0 \Mod{q}} \Phi_{c0}'(f) \Phi_c'(f), \quad \Sigma_3'=\sum_{f \equiv 0 \Mod{q}} \Phi_c'(f)^2.
    \]
    We will show
    \begin{align}
        \Sigma_1'&=M_0'+O(q^{-1}x^\e \widetilde{A}_j^2) \label{sum1},\\
        \Sigma_2'&=M_0'+O \left(q^{-1} x^{1+\e}+q^{-1}x^{\frac{1}{3}+\e}\widetilde{A}_j^2 \right) \label{sum2},\\
        \Sigma_3' &= M_0'+ O \left( q^{-1}x^{1+\e}+q^{-1}\widetilde{A}_j^{\frac{5}{2}} x^{\frac{1}{2}+\e} \left( |h|^\theta \widetilde{A}_j^\theta+x^\theta \right) \right), \label{sum3}
    \end{align}
    where $M_0'$ is the analog of the quantity $M_0$ defined in Topacogullari \cite[(5.6)]{topacogullari2018shifted}, i.e.
    \begin{align*}
        M_0'=\sum_{\substack{a_2,\ldots,a_{|I|}\\ \prod_{i=2}^{|I|} a_i \text{ squarefree}\\ (\prod_{i=2}^{|I|} a_i, \, \prod_{\{1\} \subsetneq S \subseteq I} g_S)=1}}\sum_{\substack{b_2,\ldots,b_{|I|}\\ \prod_{i=2}^{|I|} b_i \text{ squarefree}\\ (\prod_{i=2}^{|I|} b_i, \, \prod_{\{1\} \subsetneq S \subseteq I} g_S)=1}}
\iiint_{\mathbb{R}^3}
&\Delta_{\delta_1}(\xi_1+h)
 \Delta_{\delta_2}(\xi_2+h)
\\
&\times
 C'_{\delta_1,\delta_2}(h;q,\mathbf a,\mathbf b)
 \frac{
        \mathcal W_{\mathbf a}(\xi_1,\eta)
        \mathcal W_{\mathbf b}(\xi_2,\eta)
      }{\eta^2}
 \d\eta\d\xi_1\d\xi_2 ,
    \end{align*}
    where $\mathcal W_{\mathbf a}(\xi,\eta)$ is defined by
    \begin{align*}
        \mathcal W_{\mathbf a}(\xi,f)
        :=
        \frac{1}{a_2 \cdots a_{|I|}}\,
        w\left(\frac{\xi-h}{x}\right)
        v_1\left(\frac{c_1^2(\xi-h)}{fa_2 \cdots a_{|I|}}\prod_{\substack{\{1\} \subsetneq S \subseteq I }} g_S\right)
        \prod_{i=2}^{|I|} v_i \Bigg(c_i^2a_i \prod_{\substack{\{i\} \subsetneq S \subseteq I }} g_S\Bigg),
    \end{align*}
    and analogously for $\mathcal W_{\mathbf b}(\xi,\eta)$, and
    \begin{align*}
        C'_{\delta_1,\delta_2}(h;q,\mathbf a,\mathbf b)
 :=
 \sum_{d_1,d_2\geq 1}
 \frac{
        \mu^2(d_1)\mu^2(d_2)c_{d_1}(h)c_{d_2}(h)
      }{
        d_1^{1+\delta_1}d_2^{1+\delta_2}
        \operatorname{lcm}
        \left(q,\frac{d_1}{(d_1,a_2 \cdots a_{|I|})},\frac{d_2}{(d_2,b_2 \cdots b_{|I|})}\right)
      } .
    \end{align*}
    Observe that \eqref{sum1},\eqref{sum2},\eqref{sum3} together imply
    \begin{align*}
        R_j' &\ll \frac{x^{1+\e}}{q\widetilde{A}_j^{\frac12}}+q^{-1}\widetilde{A}_j^{\frac34}x^{\frac34+\e} \left( |h|^{\frac{\theta}{2}} \widetilde{A}_j^{\frac{\theta}{2}}+x^{\frac\theta2} \right),
    \end{align*}
    from which (\ref{rem3}) follows after summing over $(j_r)_{r \in I}$, $(c_i)_{i \in I}$, $(g_S)_S$ outside, which additionally contributes $\ll x^{\e}$. To prove (\ref{sum1}) and (\ref{sum2}), recall the quantity $\Phi_0$ in Topacogullari \cite{topacogullari2018shifted}, and note the quantities $\Phi_{c0}'$ and $\Phi_0$ are analogous. Therefore, the same argument that handles $\Sigma_1$ and $\Sigma_2$ in Topacogullari \cite{topacogullari2018shifted} also proves (\ref{sum1}) and (\ref{sum2}). To prove (\ref{sum3}), we write
    \[
    \Sigma_3'= \sum_{\substack{a_i,a_i', \, i \in I\\ \mu^2(\prod_{i \in I} a_i)=\mu^2(\prod_{i \in I}a_i')=1\\(\prod_{i \in I} a_ia_i',\prod_{\substack{S}} g_S)=1}} \prod_{i \in I} v_i(a_i) v_i(a_i') \Sigma_{3a}(qa_1 \cdots a_{|I|},qa_1' \cdots a_{|I|}'),
    \]
    where the product over $S$ is taken over all $S \subseteq I$ with $|S| \ge 2$, and $\Sigma_{3a}(r_1,r_2)$ is defined by
    \[
    \Sigma_{3a}(r_1,r_2)=\sum_b w \left( \frac{r_1b}{x} \right) w \left( \frac{r_2b}{x} \right) d(r_1b+h)d(r_2b+h).
    \]
    For $a_1 \cdots a_{|I|} \neq a_1' \cdots a_{|I|}'$, this corresponds to the sum considered in Theorem \ref{thm:certaindivisor}. In Topacogullari \cite{topacogullari2018shifted}, the proof instead relied on Theorem 1.1 of Topacogullari \cite{topacogullari2015certain}. Since the main term in Theorem \ref{thm:certaindivisor} is also quadratic in $\log x$, our main term coincides with theirs. Consequently, the proof of the main term in (\ref{sum3}) is analogous to that in Topacogullari \cite{topacogullari2018shifted}, and we only need to handle the error term, which we denote by $R_3$. First, we split
    \[
    R_3 \ll \sum_{\substack{r \asymp \widetilde{A}_{j}\\r \equiv 0 \Mod{q}}} d_{|I|}(r)^2 |\Sigma_{3a} (r,r)|+\sum_{\substack{a_i,a_i' \asymp \widetilde{A}_{j,i} \, \forall i \in I,\\ \prod a_i \neq \prod a_i',\\
    \prod a_i,\prod a_i' \text{ squarefree}}} |R_{3a}(qa_1 \cdots a_{|I|},qa_1 \cdots a_{|I|}')|,
    \]
    where $R_{3a}$ is the error term estimating $\Sigma_{3a}$ using Theorem \ref{thm:certaindivisor}. Trivially, the first sum is $\ll x^{1+\e}/q$. For the second sum, since everything in the summand is non-negative, first write
    \[
    \sum_{\substack{a_i,a_i' \asymp \widetilde{A}_{j,i}, \, i \in I\\ a_1 \cdots a_{|I|} \neq a_1' \cdots a_{|I|}'\\ \prod a_i,\prod a_i' \text{ squarefree}}} (\cdots) \ll \sum_{\substack{r_1,r_2 \asymp \widetilde{A}_{j}\\r_1,r_2 \text{ squarefree}\\q \mid (r_1,r_2)}} d_{|I|}(r_1)d_{|I|}(r_2) (\cdots).
    \]
    Since the divisor functions contribute at most $\widetilde{A}_{j}^\e \ll x^\e$, we ignore them. We write $r_1=r_0c_1$ and $r_2=r_0c_2$. Therefore,
    \begin{align*}
    &\sum_{\substack{r_1,r_2 \asymp \widetilde{A}_{j}\\ r_1,r_2 \text{ squarefree}\\q \mid (r_1,r_2)}} x^{\frac12+O(\eta)} r_0^\e r_2^{\frac12+\e}\gcd \left( \widetilde{r}_2-\widetilde{r}_1,r_0^\infty \right)^{1+\e}\left( \frac{|h|^\theta |r_2-r_1|^\theta}{r_0^\theta}+\frac{x^\theta}{r_0^\theta} \right)\\
    &\ll \sum_{\substack{r_0 \ll \widetilde{A}_{j}\\q \mid r_0}} \sum_{c_1,c_2 \asymp \widetilde{A}_{j}/r_0} x^{\frac12+O(\eta)}r_0^{\frac12+\e}c_2^{\frac12+\e} \gcd(c_2-c_1,r_0^\infty)^{1+\e} \left( \frac{|h|^\theta \widetilde{A}_{j}^\theta}{r_0^\theta}+\frac{x^\theta}{r_0^\theta} \right)\\
    &\ll \sum_{\substack{r_0 \ll \widetilde{A}_{j}\\q \mid r_0}} \sum_{\substack{d \mid r_0^\infty\\ d \ll \widetilde{A}_{j}/r_0}} \sum_{\substack{c_1,c_2 \asymp \widetilde{A}_{j}/r_0\\ c_1 \equiv c_2 \Mod{d}}} x^{\frac12+O(\eta)}r_0^{\frac12+\e}c_2^{\frac12+\e}d^{1+\e} \left( \frac{|h|^\theta \widetilde{A}_{j}^\theta}{r_0^\theta}+\frac{x^\theta}{r_0^\theta} \right)\\
    &\ll \sum_{\substack{r_0 \ll \widetilde{A}_{j}\\q \mid r_0}} \sum_{\substack{d \mid r_0^\infty\\ d \ll \widetilde{A}_{j}/r_0}} \sum_{c_2 \asymp \widetilde{A}_{j}/r_0} x^{\frac12+O(\eta)} r_0^{-\frac12+\e}c_2^{\frac12+\e}\widetilde{A}_{j}d^\e \left( \frac{|h|^\theta \widetilde{A}_{j}^\theta}{r_0^\theta}+x^\theta \right)\\
    &\ll \sum_{\substack{r_0 \ll \widetilde{A}_{j}\\q \mid r_0}} \sum_{\substack{d \mid r_0^\infty\\ d \ll \widetilde{A}_{j}/r_0}} x^{\frac12+O(\eta)}r_0^{-2+2\e}\widetilde{A}_{j}^{\frac52+\e}d^{\e}  \left( \frac{|h|^\theta \widetilde{A}_{j}^\theta}{r_0^\theta}+x^\theta \right).
    \end{align*}
    Using $d(r_0)=2^{\omega(r_0)}$, the inner sum is bounded by
    \begin{align*}
        \sum_{\substack{d \mid r_0^\infty\\ d \ll \widetilde{A}_{j}/r_0}} d^\e
        &\ll \widetilde{A}_{j}^\e \sum_{\substack{d \mid r_0^\infty\\ d \ll \widetilde{A}_{j}/r_0}} \frac{\widetilde{A}_{j}^\e}{d^\e} \le \widetilde{A}_{j}^{2\e} \prod_{p \mid r_0} \left( 1+p^{-\e}+p^{-2\e}+\cdots \right)
        \leq \widetilde{A}_{j}^{2\e} \prod_{p \mid r_0} \left(1-p^{-\e}\right)^{-1}
        \leq \widetilde{A}_{j}^{2\e} d(r_0)^{-\frac{\log (1-2^{-\e})}{\log 2}},
    \end{align*}
    which is $\ll_\e x^{3\e}$ by the divisor bound. Therefore,
    \begin{align*}
        R_3 &\ll_\e q^{-1}x^{1+\e}+\widetilde{A}_{j}^{\frac52+\e} x^{\frac12+O(\eta)} \left(|h|^\theta \widetilde{A}_{j}^\theta+x^\theta \right) \sum_{\substack{r_0 \ll \widetilde{A}_{j}\\q \mid r_0}} r_0^{-2+2\e}\\
        &\ll_\e q^{-1}x^{1+\e}+ q^{-1}\widetilde{A}_{j}^{\frac52+\e} x^{\frac12+O(\eta)} \left(|h|^\theta \widetilde{A}_{j}^\theta+x^\theta \right).
    \end{align*}
    This gives the required error term since $\widetilde{A}_{j} \ll x$, therefore proving \eqref{sum3} and thus (\ref{rem3}) as well.
\end{proof}
To use Lemma \ref{lem:remainderbound}, we prove that any factorisation into $A_1,\ldots,A_k$ can be partitioned into three admissible cases.
\begin{lem} \label{lem:partition}
    Let $k \ge 4$. Suppose $\alpha_1,\ldots,\alpha_k \in [0,1]$ satisfies $\alpha_1+\cdots+\alpha_k=1$ and $\alpha_1 \geq \alpha_2 \geq \cdots \geq \alpha_k$. Then, for any $\delta \leq \frac{1}{16}$, at least one of the following holds:
    \begin{enumerate}[(A).]
        \item $\alpha_1 \geq \frac13+\frac23 \delta$.
        \item $\alpha_1+\alpha_2 \geq \frac12+\delta$.
        \item There exists a non-empty index set $I \subseteq \{1,2,\ldots,k\}$ such that $2 \delta \leq \sum_{i \in I} \alpha_i \leq \frac13-\frac43 \delta$.
    \end{enumerate}
\end{lem}
\begin{proof}
    Suppose (A) and (C) do not hold. Then, $\alpha_1<\frac13+\frac23 \delta$ and
    \[
    \alpha_1+\alpha_2>\frac13-\frac43 \delta \quad \text{ or } \quad \alpha_1+\alpha_2<2 \delta.
    \]
    If $\alpha_1+\alpha_2<2\delta$, then $\alpha_1<2\delta$, and so $\alpha_i<2 \delta$ for all $i=1,2,\ldots,k$. Since $\delta<\frac{1}{16}$, we have $\left( \frac13-\frac43 \delta \right)-2\delta>2\delta.$ Since $\alpha_1+\cdots+\alpha_k=1>\frac13-\frac43 \delta$, we get a contradiction since there exists $1 \leq J \leq k$ such that $$2 \delta \leq \sum_{1 \leq j \leq J} \alpha_j \leq \frac13-\frac43 \delta.$$
    Therefore $\alpha_1+\alpha_2>\frac13-\frac43 \delta$. Note also $\alpha_2$ does not satisfy (C). If $\alpha_2<2 \delta$, then $\alpha_3,\ldots,\alpha_k<2 \delta$. Since
    $
    \alpha_2+\cdots+\alpha_k=1-\alpha_1>\frac23-\frac23 \delta>\frac13-\frac43 \delta
    $
    and $(\frac13-\frac43 \delta)-2 \delta>2 \delta$, again we have a contradiction since there exists $2 \leq J \leq k$ such that
    \[
    2 \delta \leq \sum_{2 \leq j \leq J} \alpha_j \leq \frac13-\frac43 \delta.
    \]
    Thus we must have $\alpha_2>\frac13-\frac43 \delta$, so $\alpha_1,\alpha_2>\frac13-\frac43 \delta$, so $\alpha_3+\cdots+\alpha_k<\frac13+\frac83 \delta$. If $\alpha_1 \leq \frac13-\frac13 \delta$ and $\alpha_2 \leq \frac13-\frac56 \delta$, then $\alpha_3 \leq \alpha_2$ implies $\alpha_1+\alpha_2+\alpha_3 \leq 1-2 \delta$, so $\alpha_4+\cdots+\alpha_k \geq 2\delta$. By failure of (C), we must have $\alpha_4+\cdots+\alpha_k>\frac13-\frac43 \delta$. But then this implies $\alpha_3<(\frac13+\frac83 \delta)-(\frac13-\frac43 \delta)=4\delta$, so $\alpha_3<2 \delta$ by failure of (C). Therefore, $\alpha_4,\ldots,\alpha_k<2 \delta$ but we get a contradiction as before since $\alpha_4+\cdots+\alpha_k>\frac13-\frac43 \delta.$ Thus we must have $\alpha_1>\frac13-\frac13 \delta$ or $\alpha_2>\frac13-\frac56 \delta$. We have two cases.
    \begin{itemize}
        \item Case 1: $\alpha_1>\frac13-\frac13 \delta$. Then, we are done since
        \[
        \alpha_1+\alpha_2>\left( \frac13-\frac13 \delta \right)+ \left(\frac13-\frac43 \delta \right)=\frac23-\frac53 \delta \geq \frac12+\delta
        \]
        for $\delta \leq \frac{1}{16}$.
        \item Case 2: $\alpha_2>\frac13-\frac56 \delta$. Then, $\alpha_1 \geq \alpha_2$, so we are done since
        \[
        \alpha_1+\alpha_2>2 \left( \frac13-\frac56 \delta \right)=\frac23-\frac53 \delta \geq \frac12+\delta
        \]
        for $\delta \leq \frac{1}{16}$.
    \end{itemize}
    Therefore, at least one of (A), (B) or (C) must hold, and we are done.
\end{proof}
\begin{proof}[Proof of Theorem \ref{thm:generaliseddivisor}]
    We first treat the main term. Here, our method diverges from Topacogullari \cite{topacogullari2018shifted} since their argument is not sufficient for large $k$. Let $v:(0,\infty) \to [0,\infty)$ be a smooth function compactly supported on $[1,2]$, such that
\[
\sum_{j \in \mathbb{Z}} v \left( \frac{2^ju}{x^{\frac1k}} \right)=1
\]
for any $x \in \mathbb{R}^+$, and let $v_j(u) := v(2^ju/x^\frac1k)$. Since there must exist at least one $a_i$ with $a_i \leq x^{\frac1k}$, using Lemma \ref{lem:remainderbound} without loss of generality it suffices to compute
\begin{align*}
    \sum_{j \in \mathbb{N}} \int_\mathbb{R} &w \left( \frac{u}{x} \right) \sum_{\substack{a_2,\ldots,a_k\\ d \mid a_2 \cdots a_k}} \frac{1}{a_2 \cdots a_k} v_j \left( \frac{u}{a_2 \cdots a_k} \right) \lambda_{h,d}(u+h) \d u,
\end{align*}
Let $g,g':\mathbb{N}^2 \to \mathbb{R}$ be the functions
\[
g(m,h) := \sum_{d \mid m} \frac{c_d(h)}{d}, \quad g'(m,h) := \sum_{d \mid m} \frac{c_d(h)\log d}{d},
\]
and it suffices to estimate $\sum_j (Q_{1,j}+2\gamma Q_{2,j}-2Q_{3,j})$, where
\begin{align*}
    Q_{1,j} &:=\int_{\mathbb{R}} v_j(u) \sum_m d_{k-1}(m) g(m,h) w \left( \frac{um}{x} \right) \log(um+h) \d u,\\
    Q_{2,j} &:= \int_{\mathbb{R}} v_j(u) \sum_m d_{k-1}(m)g(m,h)w \left( \frac{um}{x} \right) \d u,\\
    Q_{3,j} &:= \int_{\mathbb{R}} v_j(u) \sum_m d_{k-1}(m)g'(m,h) w \left( \frac{um}{x} \right) \d u.
\end{align*}
We treat $Q_{1,j}$ first. Let $W:(0,\infty) \to \mathbb{R}$ be given by
\[
W(\xi) := w \left( \frac{\xi}{x} \right) \log(\xi+h),
\]
and note $W$ is a smooth and compactly supported function, with Mellin transform
\[
\widetilde{W}(s)=\int_{\mathbb{R}}w \left( \frac{\xi}{x} \right) \log(\xi+h) \xi^{s-1} \d \xi.
\]
We mention that it is important for the weight function $w$ to be smooth, since we require the Mellin transform to decay rapidly. By Mellin inversion, for any $\sigma>1$ we have
\begin{align*}
    Q_{1,j} &=\frac{1}{2 \pi i}\int_{\mathbb{R}} v_j(u) \int_{\sigma-i\infty}^{\sigma+i\infty} \widetilde{W}(s)\sum_m \frac{d_{k-1}(m) g(m,h)}{u^sm^s} \d s \d u\\
    &=: \frac{1}{2 \pi i}\int_{\mathbb{R}} v_j(u) \int_{\sigma-i\infty}^{\sigma+i\infty} u^{-s} \widetilde{W}(s) D_h(s) \d s \d u.
\end{align*}
Now, the main term of $Q_{1,j}$ will be the residue of $D_h(s)$ at $s=1$, so we wish to use the residue theorem to move the line of integration. We investigate the poles of $D_h(s)$ by expanding into Euler products. First, from Schwarz and Spilker \cite{schwarz1994arithmetical}, $g$ is multiplicative and for all $\alpha \geq 1$ we have
\[
g(p^\alpha,h)=\sum_{0 \leq k \leq \alpha} \frac{c_{p^k}(h)}{p^k}.
\]
If $(p,h)=1$, then
\[
c_{p^k}(h) = \begin{cases}
    1, &\quad \text{ for }k=0,\\
    -1, &\quad \text{ for }k=1,\\
    0, &\quad \text{ for }k \geq 2,
\end{cases}
\]
therefore in this case $g(p^\alpha,h)=1-\frac1p$, and $g(1,h)=1$. If $p \mid h$, then
\[
c_{p^k}(h)=\mu \left( \frac{p^k}{(p^k,h)} \right) \frac{\varphi(p^k)}{\varphi \left( \frac{p^k}{(p^k,h)} \right)}.
\]
Therefore, for $\operatorname{Re}(s)>1$ we have
\begin{align*}
    D_h(s) &= \sum_m \frac{d_{k-1}(m) g(m,h)}{m^s}\\
    &= \prod_p \sum_{\alpha \geq0} \frac{d_{k-1}(p^\alpha)g(p^\alpha,h)}{p^{\alpha s}}\\
    &= \prod_{p \nmid h} \left( 1+\left( 1-\frac1p \right)\sum_{\alpha \geq 1} \frac{d_{k-1}(p^\alpha)}{p^{\alpha s}} \right) \underbrace{\prod_{p \mid h} \sum_{\alpha \geq0} \frac{d_{k-1}(p^\alpha)g(p^\alpha,h)}{p^{\alpha s}}}_{=: E_h(s)}\\
    &= E_h(s) \prod_{p \nmid h} \left( \frac1p+\left( 1-\frac{1}{p^s} \right)^{-(k-1)} \left( 1-\frac1p \right) \right).
\end{align*}
Here $E_h(s)$ is a finite product of non-vanishing holomorphic functions, so we focus mainly on the infinite product. Now observe
\begin{align*}
    P_h(s) := \zeta(s)^{-(k-1)}D_h(s) &= E_h(s) \prod_{p \mid h} \left( 1-\frac{1}{p^s} \right)^{k-1}\prod_{p \nmid h} \left( \frac1p\left(1-\frac{1}{p^s} \right)^{k-1}+1-\frac1p \right)\\
    &= E_h(s) \prod_{p \mid h} \left( 1-\frac{1}{p^s} \right)^{k-1} \prod_{p \nmid h} \left( 1-\frac{k-1}{p^{s+1}}+O(p^{-2s-1}) \right).
\end{align*}
Note $P_h(s)$ is holomorphic for $\operatorname{Re}(s)>1/2$, and so $D_h(s)$ has a pole of order $k-1$ at $s=1$. Using the residue theorem, for $\sigma'=\frac12+\e$ we have
\[
Q_{1,j}=\int_{\mathbb{R}} v_j(u) \operatorname{Res}_{s=1} u^{-s}\widetilde{W}(s)D_h(s) \d u+\frac{1}{2 \pi i} \int_{\mathbb{R}} v_j(u) \int_{\sigma'-i\infty}^{\sigma'+i\infty} u^{-s} \widetilde{W}(s) D_h(s) \d s \d u.
\]
Note
\[
u^{-s} \widetilde{W}(s)D_h(s)=\int_\mathbb{R} \frac1\xi w\left( \frac{\xi}{x} \right) \log(\xi+h) \left[ \left( \frac{\xi}{u} \right)^sD_h(s)  \right] \d \xi,
\]
and observe
\[
\operatorname{Res}_{s=1} \left( \frac{\xi}{u} \right)^sD_h(s)=\frac{\xi}{u} P_{k-2}(\log \xi,\log u),
\]
where $P_{k-2}(X,Y)$ is a polynomial of degree $k-2$. Thus the main term of $Q_{1,j}$ is
\[
\int_\mathbb{R} w \left( \frac{\xi}{x} \right) \log(\xi+h) \int_\mathbb{R} \frac{v_j(u)}{u}P_{k-2}(\log \xi,\log u) \d u \d \xi.
\]
If we sum over $j$ and bring the summation inside, we get the desired main term
\[
\int_\mathbb{R} w \left( \frac{\xi}{x} \right) P_k(\log x,\log \xi,\log(\xi+h)) \d \xi.
\]
It remains to bound the error term. Using decay estimates for $\widetilde{W}(s)$ and $|\zeta(\sigma'+it)| \ll |t|^{\frac12+\e}$ for $t \ge 1$, we have
\[
\left|\sum_j Q_{1,j}-\int_\mathbb{R} w \left( \frac{\xi}{x} \right) P_k(\log x,\log \xi,\log(\xi+h)) \d \xi\right| \ll x^{\frac12+\frac{1}{2k}+\e}.
\]
A very similar argument applies to $Q_{2,j}$ and $Q_{3,j}$. We omit the argument for $Q_{2,j}$, but we mention key modifications in the argument for $Q_{3,j}$ since $g'(\cdot,h)$ is not multiplicative. For $\beta \in \mathbb{R}$, let
\[
g_\beta(m,h) := \sum_{d \mid m} \frac{c_d(h)}{d^\beta},
\]
Then,
\[
g'(m,h)=-\frac{\partial}{\partial \beta} g_\beta(m,h) \Biggr|_{\beta=1},
\]
and so for $\operatorname{Re}(s)>1$, we have
\[
\widetilde{D}_h(s) := \sum_m \frac{d_{k-1}(m)g'(m,h)}{m^s} = -\frac{\partial}{\partial \beta} \left( \sum_m \frac{d_{k-1}(m)g_\beta(m,h)}{m^s} \right) \Biggr|_{\beta=1},
\]
and let $D_{h,\beta}(s)$ be the inner sum. As before, we analyse the poles of $D_{h,\beta}(s)$. Analogous to $g$, it can be shown that $g_\beta(\cdot,h)$ is multiplicative with analogous properties as $g$, in particular
\[
g_\beta(1,h)=1, \quad g_\beta(p^\alpha,h)=1-\frac{1}{p^\beta}
\]
for $\alpha \geq 1$ and $p \nmid h$. Therefore, as before we have
\[
D_{h,\beta}(s)=E_{h,\beta}(s) \prod_{p \nmid h} \left( \frac1{p^\beta}+\left( 1-\frac{1}{p^s} \right)^{-(k-1)} \left( 1-\frac{1}{p^\beta} \right) \right),
\]
where $E_{h,\beta}(s)$ is a finite product of non-vanishing holomorphic functions. Taking the logarithmic derivative, we have
\[
\frac{1}{D_{h,\beta}(s)}\frac{\partial}{\partial \beta} D_{h,\beta}(s) \Biggr|_{\beta=1}=\frac{1}{E_{h,\beta}(s)}\frac{\partial}{\partial \beta} E_{h,\beta}(s)\Biggr|_{\beta=1}-\sum_{p \nmid h}\frac{\log p}{p} \cdot \frac{1-\left( 1-\frac{1}{p^s} \right)^{-(k-1)}}{\frac1p+\left( 1-\frac{1}{p^s} \right)^{-(k-1)} \left(1-\frac1p \right)}
\]
The right hand side is analytic for $\operatorname{Re}(s)>\frac12$, so $\widetilde{D}_h(s)$ has an order $k-1$ pole at $s=1$. The rest of the argument follows analogously to above, contributing a degree $k-1$ polynomial. Combining all contributions from $Q_{1,j},Q_{2,j},Q_{3,j}$, we proved the main term is
    \[
    \int_\mathbb{R} w \left( \frac{\xi}{x} \right) P_k(\log x,\log \xi,\log(\xi+h)) \d \xi+O_{w,\e,k}(x^{\frac12+\frac{1}{2k}+\e}).
    \]
To treat the required error term from Lemma \ref{lem:remainderbound}, the idea is to dyadically split the sum over $a_1 \cdots a_k \asymp x$ to a sum over $a_i \asymp A_i$ with $A_1 \cdots A_k \asymp X$, and we use the partition in Lemma \ref{lem:partition} along with (\ref{rem1}), (\ref{rem2}), and (\ref{rem3}). We set $A_1 \geq A_2 \geq \cdots \geq A_k$, and let
    \[
    X_1=x^{\frac13+\frac23 \delta}, \quad X_2=x^{\frac12+\delta}, \quad X_3=x^{\frac13-\frac43 \delta}, \quad X_4=x^{2 \delta}.
    \]
    If $A_1 \gg X_1$, then use (\ref{rem1}) to get
    \[
    R_{v_1} \ll x^{1-\delta+\e}.
    \]
    If $A_1A_2 \gg X_2$, then use (\ref{rem2}) to get
    \[
    R_{v_1} \ll x^{1-\delta+\e}\left( 1+x^{2 \theta \delta} \right)\left(1+1 \right) \left( 1+\frac{|h|^{\frac{1}{4}}}{x^{\frac14-\frac12 \delta}} \right) \ll x^{1-\delta+2 \delta \theta+\e}\left( 1+\frac{|h|^{\frac{1}{4}}}{x^{\frac14-\frac12 \delta}} \right).
    \]
    Otherwise, there is a non-empty index set $I \subseteq \{1,2,\ldots,k\}$ such that $X_4 \ll \prod_{i \in I} A_i \ll X_3$. For convenience let $A=\prod_{i \in I} A_i$, so using (\ref{rem3}), we get
    \begin{align*}
        R_{v_1} &\ll \frac{x^{1+\e}}{X_4^{\frac12}}+X_3^{\frac34} x^{\frac34+\e}\left( |h|^{\frac{\theta}{2}}X_3^{\frac{\theta}{2}}+{x^{\frac{\theta}{2}}} \right)\\
        &=x^{1-\delta+\e}+x^{1-\delta+\e}(|h|^{\frac{\theta}{2}} x^{\frac{\theta}{6}-\frac{2\delta}{3} \theta}+x^{\frac{\theta}{2}})\\
        &\ll x^{1-\delta+\frac{\theta}{2}+\e} \left(1+\frac{|h|^{\frac{\theta}{2}}}{x^{\frac{\theta}{3}+\frac{2 \delta}{3} \theta}}\right).
    \end{align*}
    Combining the above bounds, we are done.
\end{proof}
\begin{cor} \label{cor:generaliseddivisor}
    Let $k \ge 4$, $w:[1/2,1] \to [0,\infty)$ be a smooth and compactly supported function, and $x \in \mathbb{R}^+$. Then for $0 \leq \delta \leq \frac{1}{16}$ we have the following estimates for different ranges of $h \in \mathbb{N}$.
    \begin{enumerate}
        \item For $|h| \leq x^{\frac23+\frac{4 \delta}{3}}$, we have
        \[
        \sum_n w \left( \frac{n}{x} \right) d_k(n) d(n \pm h)=xP_{k,h,w}(\log x)+O_{k,w,\delta,\e}(x^{1-\delta+\frac{\theta}{2}+\e}),
        \]
        where $P_{k,h,w}$ is a polynomial of degree $k$ depending only on $k,h,w$.
        \item For $x^{\frac23+\frac{4 \delta}{3}} \le |h| \le x^{1-\e}$, we have
        \[
        \sum_n w \left( \frac{n}{x} \right) d_k(n) d(n \pm h)=\int_\mathbb{R} w \left( \frac{\xi}{x} \right) P_{k,h}(\log x,\log \xi,\log(\xi+h)) \d \xi+O_{k,w,\delta,\e}(|h|^{\frac{\theta}{2}}x^{1-\delta+\frac{\theta}{6}-\frac{2 \delta}{3}\theta+\e}),
        \]
        where $P_{k,h}$ is a polynomial of degree $k$ depending only on $k$ and $h$.
    \end{enumerate}
\end{cor}
We can choose $\delta=1/16$ and $\theta \leq 7/64$ from Kim and Sarnak \cite{kim2003functoriality}. If we have $\theta=0$ from the Ramanujan-Petersson conjecture, then all ranges above can be used and are all non-trivial.
\begin{cor} \label{cor:generaliseddivisorsmall}
    Let $k \ge 4$, $w:[1/2,1] \to [0,\infty)$ be a smooth and compactly supported function, and $x \in \mathbb{R}^+$. Then we have the following estimates for different ranges of $h \in \mathbb{Z}$.
    \begin{enumerate}[(a)]
        \item For $|h| \leq x^{\frac34}$, we have
        \[
        \sum_n w \left( \frac{n}{x} \right) d_k(n) d(n + h)=xP_{k,h,w}(\log x)+O_{k,w,\e}(x^{\frac{15}{16}+\frac12 \theta+\e}),
        \]
        where $P_{k,h,w}$ is a polynomial of degree $k$ depending only on $k,h,w$.
        \item For $x^{\frac34} \leq |h| \leq x^{1-\e}$, we have
        \[
        \sum_n w \left( \frac{n}{x} \right) d_k(n) d(n + h)=\int_\mathbb{R} w \left( \frac{\xi}{x} \right) P_{k,h}(\log x,\log \xi,\log(\xi+h)) \d \xi+O_{k,w,\e}(|h|^{\frac{\theta}{2}}x^{\frac{15}{16}+\frac{1}{8} \theta+\e}),
        \]
        where $P_{k,h}$ is a polynomial of degree $k$ depending only on $k$ and $h$.
    \end{enumerate}
\end{cor}
Using $\theta \leq 7/64$ from Kim and Sarnak \cite{kim2003functoriality}, the bound from $|h| \ge x^{\frac34}$ is only nontrivial when $|h| \leq x^{\frac{25}{28}-\e}$. Since $\frac{15}{19}<\frac{25}{28}$, Theorem \ref{thm:intromainthm} is an improvement to Topacogullari \cite[Theorem 1.2]{topacogullari2018shifted} when $k$ is large.
\setcounter{section}{1}
\setcounter{thm}{0}
\begin{thm}
    Let $k \ge 4$, and $w:[1/2,1] \to [0,\infty)$ be smooth and compactly supported. Let $h$ be a non-zero integer such that $|h| \ll x^{\frac{25}{28}-\eta}$ for some $0<\eta<25/28$. Then,
    \[
    \sum_n w\left( \frac{n}{x} \right) d_k(n) d(n + h)=xP_{k,h,w}(\log x)+O_{k,w,\e}(x^{1-\frac{7}{128} \eta+\e}+x^{\frac{127}{128}+\e}),
    \]
    where $P_{k,h,w}$ is a polynomial of degree $k$ depending only on $k,h,w$.
\end{thm}
\setcounter{section}{6}

\vfill
\setcitestyle{numbers}
\bibliographystyle{apalike}
\bibliography{bibliography.bib}

@article{grimmelt2024twisted,
  title={Twisted correlations of the divisor function via discrete averages of $\operatorname{SL}_2 (\mathbb{R})$ Poincar\'{e} series},
  author={Grimmelt, Lasse and Merikoski, Jori},
  journal={arXiv preprint arXiv:2404.08502},
  year={2024}
}

@article{topacogullari2018shifted,
  title={The shifted convolution of generalized divisor functions},
  author={Topacogullari, Berke},
  journal={International Mathematics Research Notices},
  volume={2018},
  number={24},
  pages={7681--7724},
  year={2018},
  publisher={Oxford University Press}
}

@article{topacogullari2015certain,
  title={On a certain additive divisor problem},
  author={Topacogullari, Berke},
  journal={Acta Arithmetica},
volume={181},
  number={2},
  year={2015}
}

@article{kim2003functoriality,
  title={Functoriality for the exterior square of $gl_4$ and the symmetric fourth of $gl_2$},
  author={Kim, Henry and Sarnak, Peter},
  journal={Journal of the American Mathematical Society},
  volume={16},
  number={1},
  pages={139--183},
  year={2003}
}

@book{schwarz1994arithmetical,
  title={Arithmetical functions},
  author={Schwarz, Wolfgang and Spilker, J{\"u}rgen},
  volume={184},
  year={1994},
  publisher={Cambridge University Press}
}

@article{linnik1963dispersion,
  title={The dispersion method in binary additive problems},
  author={Linnik, I︠U︡ and Schuur, S},
  journal={American Mathematical Society, Providence, R.I., 1963},
  year={1963}
}

@inproceedings{motohashi1994binary,
  title={The binary additive divisor problem},
  author={Motohashi, Yoichi},
  booktitle={Annales scientifiques de l'Ecole normale sup{\'e}rieure},
  volume={27},
  number={5},
  pages={529--572},
  year={1994}
}

@article{topacogullari2016shifted,
  title={The shifted convolution of divisor functions},
  author={Topacogullari, Berke},
  journal={The Quarterly Journal of Mathematics},
  volume={67},
  number={2},
  pages={331--363},
  year={2016},
  publisher={Oxford University Press}
}

@article{motohashi1980asymptotic,
  title={An asymptotic series for an additive divisor problem},
  author={Motohashi, Yoichi},
  journal={Mathematische Zeitschrift},
  volume={170},
  number={1},
  pages={43--63},
  year={1980},
  publisher={Springer}
}

@article{drappeau2017sums,
  title={Sums of Kloosterman sums in arithmetic progressions, and the error term in the dispersion method},
  author={Drappeau, Sary},
  journal={Proceedings of the London Mathematical Society},
  volume={114},
  number={4},
  pages={684--732},
  year={2017},
  publisher={Wiley Online Library}
}
\end{document}